\documentclass[a4paper,10pt]{amsart}

\usepackage{amsthm,amsmath,amssymb,amsfonts}
\usepackage{libertine}
\usepackage{libertinust1math}
\usepackage{natbib}
\usepackage{tikz-cd}
\usepackage{appendix}
\usepackage{stmaryrd}
\usepackage{quiver}
\usepackage[T1]{fontenc}
\usepackage[utf8]{inputenc}
\usepackage{verbatim}
\usepackage{marvosym}
\usepackage{stackrel}
\usepackage{graphicx}
\usepackage{eucal}
\usepackage{bbold}
\usepackage{lipsum}
\usepackage{svg}
\usepackage{bbm}
\usepackage{caption}
\usepackage{enumerate}
\usepackage{mathtools}
\usepackage{epigraph}
\usepackage{color}
\usepackage{comment}
\usepackage{microtype}
\usepackage{relsize}
\usepackage{amsfonts}
\usepackage{adjustbox}
\usepackage{subfig}
\usepackage{framed}
\usepackage{hyperref}
\usepackage{orcidlink}
\usepackage[capitalise]{cleveref}
\usepackage{trimclip}

\DeclareFontFamily{U}{min}{}
\DeclareFontShape{U}{min}{m}{n}{<-> udmj30}{}

\tikzcdset{pullback/.style = {"\lrcorner"{anchor = center, pos = 0.125}, draw = none}}
\tikzcdset{pushout/.style = {"\lrcorner"{anchor = center, rotate=180, pos = 0.125}, draw = none}}

\newtheorem{theorem}{Theorem}[section]
\newtheorem*{theorem*}{Theorem}

\newtheorem{lemma}[theorem]{Lemma}

\newtheorem{corollary}[theorem]{Corollary}

\newtheorem{proposition}[theorem]{Proposition}

\theoremstyle{definition}
\newtheorem{example}[theorem]{Example}
\newtheorem{definition}[theorem]{Definition}
\newtheorem{remark}[theorem]{Remark}

\newcommand{\mc}[1]{\mathcal{#1}}
\newcommand{\mb}[1]{\mathbf{#1}}
\newcommand{\mbb}[1]{\mathbb{#1}}

\newcommand{\rg}{\mc U}

\newcommand{\mr}[1]{\mathrm{#1}}

\newcommand{\ms}[1]{\mathsf{#1}}
\newcommand{\Ind}{\mathbf{Ind}}
\newcommand{\Pro}{\mathbf{Pro}}
\newcommand{\Set}{\mb{Set}}

\newcommand{\sub}{\mr{Sub}}

\newcommand{\Lex}{\mb{Lex}}
\newcommand{\Pos}{\mb{Pos}}
\newcommand{\Kp}{\mb{Kp}}

\newcommand{\IKp}{\mb{IKp}}

\newcommand{\Kpp}{\Kp_{*}}

\newcommand{\DL}{\mb{DL}}

\newcommand{\HA}{\mb{HA}}

\newcommand{\sh}{\mb{Sh}}
\newcommand{\psh}{\mb{Psh}}

\newcommand{\op}{^{\mathrm{op}}}
\newcommand{\inv}{^{\mathrm{-1}}}
\newcommand{\pf}[1]{\widehat{#1}}
\newcommand{\qsi}[1]{\widetilde{#1}}

\newcommand{\geo}[1]{\left|#1\right|}
\newcommand{\ov}[1]{\overline{#1}}
\newcommand{\set}[1]{\{\,#1\,\}}
\newcommand{\eff}{\Leftrightarrow}
\newcommand{\conjt}{\;\&\;}
\newcommand{\pair}[1]{\left\langle#1\right\rangle}

\newcommand{\elem}{\int\!\!}
\newcommand{\nt}{\Rightarrow}
\newcommand{\scomp}[2]{\{\,#1\mid#2\,\}}
\newcommand{\yon}{\mathtt{y}}
\newcommand{\surj}{\twoheadrightarrow}
\newcommand{\inj}{\rightarrowtail}
\newcommand{\hook}{\hookrightarrow}

\newcommand{\fn}{_{\mr{f.}}}

\newcommand{\dv}{\operatorname{\uparrow}}

\newcommand{\fa}[2]{\forall #1\!\in\! #2.\ }
\newcommand{\ex}[2]{\exists #1\!\in\! #2.\ }

\newcommand{\uv}[1]{\underline{#1}}

\newcommand{\spec}{\operatorname{Spec}}

\newcommand{\kspec}{\operatorname{\ov{Spec}}}

\newcommand{\lan}{\ms{lan}}

\DeclareFontFamily{U}{dmjhira}{}
\DeclareFontShape{U}{dmjhira}{m}{n}{ <-> dmjhira }{}
\DeclareRobustCommand{\yon}{\text{\usefont{U}{dmjhira}{m}{n}\symbol{"48}}}

\newcommand{\um}{\Omega}

\newcommand{\plo}{\!+\!1}

\newcommand{\Stone}{\mb{Stone}}
\newcommand{\KHaus}{\mb{KHaus}}

\newcommand{\exreg}{_{\mr{ex/reg}}}

\newcommand{\Forall}{\operatorname{\ms A}}
\newcommand{\Exists}{\operatorname{\ms E}}

\makeatletter
\newcommand{\ct@}[2]{%
  \vtop{\m@th\ialign{##\cr
    \hfil$#1\operator@font lim$\hfil\cr
    \noalign{\nointerlineskip\kern1.5\ex@}#2\cr
    \noalign{\nointerlineskip\kern-\ex@}\cr}}%
}
\newcommand{\ct}{%
  \mathop{\mathpalette\ct@{\rightarrowfill@\textstyle}}\nmlimits@
}
\makeatother
\makeatletter
\newcommand{\lt@}[2]{%
  \vtop{\m@th\ialign{##\cr
    \hfil$#1\operator@font lim$\hfil\cr
    \noalign{\nointerlineskip\kern1.5\ex@}#2\cr
    \noalign{\nointerlineskip\kern-\ex@}\cr}}%
}
\newcommand{\lt}{%
  \mathop{\mathpalette\lt@{\leftarrowfill@\textstyle}}\nmlimits@
}
\makeatother

\title{Coexact completion of profinite Heyting algebras and uniform interpolation}

\author{Lingyuan Ye \orcidlink{0000-0002-8983-0099}}

\address{
Lingyuan \textsc{Ye} \newline
Department of Computer Science and Technology\newline
University of Cambridge\newline
Cambridge, UK\newline
}

\begin{document}

\begin{abstract} 
    This paper shows that the sheaf representation of finitely generated free Heyting algebras constructed in~\cite{ghilardi2013sheaves} can be factored as the \emph{profinite completion} of Heyting algebras, followed by identifying the dual category of profinite Heyting algebras as a full subcategory of a sheaf topos. We show that the dual category of profinite Heyting algebras is an infinitary extensive regular category, and its ex/reg-completion is exactly the aforementioned sheaf topos, which we refer to as the K-topos. We show how certain properties of uniform interpolation can be generalised to the context of arbitrary profinite Heyting algebras, and that they are consequences of the \emph{internal logic} of the K-topos. Along the way we also establish various topos-theoretic properties of the K-topos.    
    
\end{abstract}

\maketitle

   {
   \hypersetup{linkcolor=black}
   \tableofcontents
   }

\section{Introduction}\label{sec:intro}

\cite{pitts1992uniform} establishes an interpretation of \emph{second-order} propositional quantifiers into intuitionistic propositional logic (IPC), and this result is now called \emph{uniform interpolation}. Let $H[n]$ be the free Heyting algebra over $n$ generators. From an algebraic perspective, uniform interpolation means the canonical inclusion $\iota : H[n]\inj H[n\plo]$ between finitely generated free Heyting algebras admits both a left and right adjoint, 
\[\begin{tikzcd}
	{H[n]} & {H[n\plo]}
	\arrow[""{name=0, anchor=center, inner sep=0}, "\iota"{description}, tail, from=1-1, to=1-2]
	\arrow[""{name=1, anchor=center, inner sep=0}, "\Forall", curve={height=-18pt}, from=1-2, to=1-1]
	\arrow[""{name=2, anchor=center, inner sep=0}, "\Exists"', curve={height=18pt}, from=1-2, to=1-1]
	\arrow["\dashv"{anchor=center, rotate=-90}, draw=none, from=0, to=1]
	\arrow["\dashv"{anchor=center, rotate=-90}, draw=none, from=2, to=0]
\end{tikzcd}\]

The original proof presented in~\cite{pitts1992uniform} uses a careful analysis of an efficiently terminating sequent calculus for IPC due to~\cite{hudelmaier1992bounds} and~\cite{dyckhoff1992contraction}. Later on,~\cite{ghilardi1995sheaf} proves this result via a sheaf-theoretic representation of finitely presented Heyting algebras, and~\cite{visser1996uniform} gives a proof via bisimulation between Kripke semantics of IPC. More recently,~\cite{gool2018open} establishes this result by Esakia duality, showing maps between finitely presented Heyting algebras induce open maps between Esakia spaces. However, besides the different choices of framework -- sheaf theory, Kripke semantics, topology -- the proofs of uniform interpolation in these semantic approaches all crucially rely on a certain combinatorial fact about the existence of certain extensions of finite posets, which can be seen in all the aforementioned proofs besides Pitts' original proof-theoretic approach.

The original motivation of this paper is to obtain an \emph{algebraic} understanding of uniform interpolation. Our original proof strategy is quite na\"ive. Part of the difficulty of showing such adjoints exist lies in the fact that the free Heyting algebras $H[n]$ are not complete for $n \ge 2$. Thus, \emph{a priori}, one cannot use infinite joins or meets to define the adjoints. This way, if we can \emph{embed} these Heyting algebras into complete ones where the adjoints do exist, then we can prove uniform interpolation by showing that the relevant adjoints restrict to these free Heyting algebras. 

One of the natural choices is to embed them into their \emph{profinite completions}. For any finite $n$, there is a natural map $H[n] \inj \pf{H[n]}$ from the free Heyting algebra on $n$ generators to its profinite completion. This map is in fact an \emph{embedding} by the finite model property of IPC. Now $\pf{H[n]}$ is a complete lattice, and the natural extension $\pf{H[n]} \inj \pf{H[n\plo]}$ preserves both arbitrary meets and joins, thus admits both a left and right adjoint,
\[\begin{tikzcd}[row sep = large]
	{H[n]} & {H[n\plo]} \\
	{\pf{H[n]}} & {\pf{H[n\plo]}}
	\arrow[""{name=0, anchor=center, inner sep=0}, tail, from=1-1, to=1-2]
	\arrow[tail, from=1-1, to=2-1]
	\arrow[""{name=1, anchor=center, inner sep=0}, curve={height=-18pt}, dashed, from=1-2, to=1-1]
	\arrow[""{name=2, anchor=center, inner sep=0}, curve={height=18pt}, dashed, from=1-2, to=1-1]
	\arrow[tail, from=1-2, to=2-2]
	\arrow[""{name=3, anchor=center, inner sep=0}, tail, from=2-1, to=2-2]
	\arrow[""{name=4, anchor=center, inner sep=0}, "\Forall", curve={height=-18pt}, from=2-2, to=2-1]
	\arrow[""{name=5, anchor=center, inner sep=0}, "\Exists"', curve={height=18pt}, from=2-2, to=2-1]
	\arrow["\dashv"{anchor=center, rotate=-90}, draw=none, from=0, to=1]
	\arrow["\dashv"{anchor=center, rotate=-90}, draw=none, from=2, to=0]
	\arrow["\dashv"{anchor=center, rotate=-90}, draw=none, from=3, to=4]
	\arrow["\dashv"{anchor=center, rotate=-90}, draw=none, from=5, to=3]
\end{tikzcd}\]
Hence, the aim now is to show the adjoints $\Exists,\Forall : \pf{H[n\plo]} \to \pf{H[n]}$ on the level of profinite Heyting algebras restrict to maps between the generating free Heyting algebras $\Exists,\Forall : H[n\plo] \to H[n]$. 

However, when exploring this approach we have ultimately realised that, despite the apparent differences, the profinite completion is \emph{exactly} what the sheaf representation in~\cite{ghilardi1995sheaf} is implicitly performing. We call the sheaf topos constructed in \emph{loc.\ cit.}\ the \emph{K-topos}, and denote it as $\mc K$.\footnote{The reason for this name is that the generating site for $\mc K$ is the category of Kripke frames (cf.~\cref{sec:K-topos}).} To briefly summarise the connection, recall that~\cite{ghilardi1995sheaf} essentially constructs a functor from the dual category of Heyting algebras to the K-topos,
\[ \kspec : \HA\op \to \mc K, \]
such that for $A \in \HA$, there is a map from $A$ into the subobject lattice of its image in $\mc K$,
\[ A \to \sub_{\mc K}(\kspec A). \]
When $A$ is finitely presented, it is shown in \emph{loc.\ cit.}\ that this map is an embedding via the finite model property of IPC, and that its image can be characterised via a certain bisimulation game. In this case, \emph{loc.\ cit.} proves uniform interpolation by showing that the left and right adjoints of pullback functors of subobjects in $\mc K$ restrict to these sub-Heyting algebras,
\[\begin{tikzcd}[row sep = large]
	{H[n]} & {H[n\plo]} \\
	{\sub_{\mc K}(\kspec H[n])} & {\sub_{\mc K}(\kspec H[n\plo])}
	\arrow[""{name=0, anchor=center, inner sep=0}, tail, from=1-1, to=1-2]
	\arrow[tail, from=1-1, to=2-1]
	\arrow[""{name=1, anchor=center, inner sep=0}, curve={height=-18pt}, dashed, from=1-2, to=1-1]
	\arrow[""{name=2, anchor=center, inner sep=0}, curve={height=18pt}, dashed, from=1-2, to=1-1]
	\arrow[tail, from=1-2, to=2-2]
	\arrow[""{name=3, anchor=center, inner sep=0}, "\pi\inv"{description}, tail, from=2-1, to=2-2]
	\arrow[""{name=4, anchor=center, inner sep=0}, "\forall_{\pi}", curve={height=-18pt}, from=2-2, to=2-1]
	\arrow[""{name=5, anchor=center, inner sep=0}, "\exists_{\pi}"', curve={height=18pt}, from=2-2, to=2-1]
	\arrow["\dashv"{anchor=center, rotate=-90}, draw=none, from=0, to=1]
	\arrow["\dashv"{anchor=center, rotate=-90}, draw=none, from=2, to=0]
	\arrow["\dashv"{anchor=center, rotate=-90}, draw=none, from=3, to=4]
	\arrow["\dashv"{anchor=center, rotate=-90}, draw=none, from=5, to=3]
\end{tikzcd}\]
where $\pi = \kspec \iota : \kspec H[n\plo] \to \kspec H[n]$ is the image of $\iota : H[n] \inj H[n\plo]$ under the representation functor $\kspec$.

We will show that the map $A \to \sub_{\mc K}(\kspec A)$ for a general Heyting algebra is exactly the canonical map associated to its \emph{profinite completion}, thus our original proof strategy via profinite completion is essentially the same as the sheaf representation proof. Categorically, this can be observed by factorising the $\kspec$ functor as the following composition of adjunctions, 
\[\begin{tikzcd}
	{\mc K} & {\Pro(\HA\fn)\op} & {\HA\op}
	\arrow[""{name=0, anchor=center, inner sep=0}, "{\sub_{\mc K}(-)}", curve={height=-18pt}, from=1-1, to=1-2]
	\arrow[""{name=1, anchor=center, inner sep=0}, hook', from=1-2, to=1-1]
	\arrow[""{name=2, anchor=center, inner sep=0}, curve={height=-18pt}, tail, from=1-2, to=1-3]
	\arrow["\kspec", curve={height=-18pt}, from=1-3, to=1-1]
	\arrow[""{name=3, anchor=center, inner sep=0}, "{\pf-}"{description}, from=1-3, to=1-2]
	\arrow["\dashv"{anchor=center, rotate=-90}, draw=none, from=0, to=1]
	\arrow["\dashv"{anchor=center, rotate=-90}, draw=none, from=2, to=3]
\end{tikzcd}\]
Here the dual category $\Pro(\HA\fn)\op$ of profinite Heyting algebras is a reflective subcategory of $\mc K$, where the reflection is exactly given by taking the subobject lattice; the adjunction between $\HA$ and $\Pro(\HA\fn)$ consists of the forgetful functor and the profinite completion. Thus, the counit of the composite adjunction is exactly the profinite completion,
\[ A \to \sub_{\mc K}(\kspec A) \cong \pf A. \]
In fact, we will also observe that the sheaf representation or profinite completion is also equivalent to Bellissima's representation of Heyting algebras given in~\cite{bellissima1986finitely}. These results will be discussed in~\cref{sec:nerveonK}, where we show they are all special cases of a general \emph{nerve construction} over the K-topos $\mc K$.

Once one realises the connection between profinite completion and the sheaf representation of Heyting algebras, it is natural to study how properties of profinite Heyting algebras, or in fact the K-topos, relate to uniform interpolation. In fact, there are two natural generalisations of properties of uniform interpolation: 
\begin{itemize}
    \item It follows from this observation that the adjoints for uniform interpolation are obtained via restriction of adjoints of maps between profinite Heyting algebras. Thus, one may naturally study uniform interpolation by studying these more general adjoints. In fact, most of the known logical properties of uniform interpolants as second-order propositional quantifiers hold more generally for the adjoints between profinite Heyting algebras. 
    \item Via the embedding $\Pro(\HA\fn)\op \hook \mc K$, we can in fact \emph{internalise} the adjoints between profinite Heyting algebras as actual maps between the corresponding objects in $\mc K$. One may then ask whether the \emph{external properties} of these adjoints hold \emph{internally} in the K-topos $\mc K$. 
\end{itemize}
We emphasise that the internal properties in the K-topos are the strongest form. Properties about adjoints between profinite Heyting algebras will follow from the externalisation of these internal properties, and they in turn imply properties of uniform interpolation via restriction to the generating finitely generated free Heyting algebras. 

With this perspective, in~\cref{sec:internaluniform} we will study two important and well-known properties of uniform interpolation for IPC:
\begin{enumerate}
    \item The left uniform interpolant $\Exists$ satisfies a dual Frobenius property, which is also known in constructive mathematics as the principle of \emph{independence of premise}:
    \[ \Exists(\iota\varphi \to \psi) = \varphi \to \Exists\psi. \]
    \item The right uniform interpolant $\Forall$ preserves finite joins,
    \[ \Forall(\bigvee_{i\in n}\varphi_i) = \bigvee_{i\in n}\Forall\varphi_i. \]
\end{enumerate}
We will show that they follow from the internal properties of $\mc K$ that (1) propositions in $\mc K$ are \emph{internally projective} and the subobject classifier $\Omega$ is \emph{internally injective} (cf.~\cref{internallyinjective}); and (2) the subobject classifier $\Omega$ is \emph{internally irreducible} (cf.~\cref{irreduciblesubobj}). Not only does it provide an alternative understanding of these logical principles of uniform interpolation as geometric properties of sheaves in the K-topos, but as a corollary it also implies that these two properties hold much more generally for all relevant adjoints between profinite Heyting algebras (cf.~\cref{dualfrobenius,rightadjlocal}).

This way, a general investigation of profinite Heyting algebras and the K-topos is desirable for studying uniform interpolation of IPC. From a categorical perspective, we further show that the category of profinite Heyting algebras has a much closer connection with the K-topos, besides being a reflective subcategory. The main categorical characterisation can be summarised as follows:

\begin{theorem*}[\ref{IKppreservescoproducts}, \ref{IKppreservesregularepi}, \ref{IKppreservefiltercolim}, \ref{maintheorem}]
    The reflective subcategory $\Pro(\HA\fn)\op \hook \mc K$ satisfies the following properties:
    \begin{itemize}
        \item The embedding is accessible, i.e. it preserves filtered colimits; 
        \item The subcategory is closed under small coproducts and image factorisation;
        \item This embedding identifies $\mc K$ as the ex/reg-completion of $\Pro(\HA\fn)\op$.
    \end{itemize}
\end{theorem*}

These results in particular imply various exactness properties of profinite Heyting algebras. For instance, $\Pro(\HA\fn)\op$ will be regular. Dually, this means injections between profinite Heyting algebras are closed under pushouts, or in other words profinite Heyting algebras have the strong amalgamation property. $\Pro(\HA\fn)$ will also be infinitary coextensive. Dually, this means pushouts of profinite Heyting algebras preserve arbitrary products. In fact, $\Pro(\HA\fn)\op$ as a full subcategory contains the subobject classifier of $\mc K$. Dually, this means profinite Heyting algebras have a classifier of quotient objects. In fact, from a categorical perspective, $\Pro(\HA\fn)\op$ is quite close to being a topos. The only missing part is effective quotients. Though $\Pro(\HA\fn)\op$ is cocomplete (as it will be locally finitely presentable), the quotients of equivalence relations in $\Pro(\HA\fn)\op$ are not all \emph{effective} (cf.~\cref{notexact}). If we formally add effective quotients to $\Pro(\HA\fn)\op$, we will get the K-topos $\mc K$.

The structure of this paper is outlined as follows. In~\cref{sec:K-topos} we will briefly recall the construction of the K-topos. \cref{sec:property-K} will establish some topos-theoretic properties of the K-topos. \cref{sec:nerveonK} will describe the nerve construction over the K-topos, and through this explain the connection with the sheaf representation of Heyting algebras. \cref{sec:internaluniform} will investigate the internal language of $\mc K$, and use that to establish properties of adjoints of maps between profinite Heyting algebras. \cref{sec:subtopoi} will slightly digress from the main topic and classify all subtopoi of $\mc K$. \cref{sec:exactcomp} will prove various exactness properties of the category of profinite Heyting algebras, and furthermore show $\mc K$ is the ex/reg-completion of the dual category of profinite Heyting algebras.

\section{The K-topos as sheaves on Kripke frames}\label{sec:K-topos}

Let $\Pos\fn$ be the category of finite posets. It has already been established in~\cite{birkhoff1937rings} that there is a duality between finite distributive lattices and finite posets,
\[ \spec : \DL\fn\op \simeq \Pos\fn : \mc U, \]
where $\spec D = \DL(D,2)$ is the poset of models of a finite distributive lattice, and $\mc UP$ takes a finite poset $P$ to the distributive lattice of upward closed subsets of $P$. 

Each finite distributive lattice is in fact a complete lattice, and distributivity implies it will be a Heyting algebra (also a co-Heyting algebra). In this case, it is well-known that a map between finite posets $f : P \to Q$ dually corresponds to a Heyting algebra morphism iff it is \emph{open} (see e.g.~\cite[Thm. 2.1]{ghilardi2013sheaves}): for any $p\in P$ and $fp \le q$ in $Q$, there exists $p\le p'$ that $fp' = q$. We let $\Kp$ be the category of finite posets and open maps between them. This way, we have a duality
\[ \spec : \HA\fn\op \simeq \Kp : \mc U, \]
where $\HA\fn$ is the category of finite Heyting algebras. Note that since $\HA\fn$ has finite limits, the dual category $\Kp$ has finite colimits.

We thus view $\Kp$ as the category of Kripke frames. The topos $\mc K$ will be the category of sheaves on $\Kp$ for some suitable topology. To construct this topology, we first establish some categorical properties of $\Kp$.

\begin{lemma}\label{creatingfincolimkp}
  The inclusion $\Kp \inj \Pos\fn$ preserves finite colimits.
\end{lemma}
\begin{proof}
  From a categorical perspective, the forgetful functor from Heyting algebras to distributive lattices $\HA\to\DL$ is a right adjoint (cf.~\cite{lawvere1963functorial}). Hence, when restricted to finite algebras, $\HA\fn \inj \DL\fn$ preserves finite limits, which implies that dually the inclusion $\Kp \inj \Pos\fn$ preserves finite colimits. However, for future reference an explicit proof will also be useful.

  $\Kp \inj \Pos\fn$ evidently preserves finite coproducts. For coequalisers, suppose we have a parallel pair $f_0,f_1 : R \to P$. We define an equivalence relation $\theta$ on $P$ which is the smallest one containing $\theta_0 \coloneq \scomp{(f_ir,f_jr)}{r\in R \conjt i,j\in\set{0,1}}$. We first show that for any $p\theta q$, if $p \le p'$ then there exists $q \le q'$ that $p'\theta q'$. By definition, $p\theta q$ iff we have $r_1,\cdots,r_n$ such that $p = f_{i_1}r_1$, $f_{j_k}r_k = f_{i_{k\plo}}r_{k\plo}$ and $f_{j_n}r_n = q$ for some $i_k,j_k \in \set{0,1}$ as shown by the bottom row below,
    \[\begin{tikzcd}
    	{p' = f_{i_1}r_1'} & {f_{j_1}r_1' = f_{i_2}r_2'} & {f_{j_2}r_2' = \cdots} & {\cdots = f_{j_n}r_n' =: q'} \\
    	{p = f_{i_1}r_1} & {f_{j_1}r_1 = f_{i_2}r_2} & {f_{j_2}r_2 = \cdots} & {\cdots = f_{j_n}r_n = q}
    	\arrow["\theta_0", dashed, tail reversed, from=1-1, to=1-2]
    	\arrow["\theta_0", dashed, tail reversed, from=1-2, to=1-3]
    	\arrow["\theta_0", dashed, tail reversed, from=1-3, to=1-4]
    	\arrow[from=2-1, to=1-1]
    	\arrow["\theta_0"', tail reversed, from=2-1, to=2-2]
    	\arrow[dashed, from=2-2, to=1-2]
    	\arrow["\theta_0"', tail reversed, from=2-2, to=2-3]
    	\arrow[dashed, from=2-3, to=1-3]
    	\arrow["\theta_0"', tail reversed, from=2-3, to=2-4]
    	\arrow[dashed, from=2-4, to=1-4]
    \end{tikzcd}\]
  Now given $p \le p'$, by $f_{i_1}$ being open, there is $r_1'$ that $p' = f_{i_1}r_1'$. Hence $f_{i_2}r_2 = f_{j_1}r_1 \le f_{j_1}r_1'$, which implies we can find $r_2'$ that $f_{i_2}r_2' = f_{j_1}r_1'$, and so on and so forth. Thus, we can define $q'$ to be $f_{j_n}r_n'$, and by construction $p'\theta q'$ and $q \le q'$. 
  
  This way, we can define a unique order on equivalence classes of $\theta$,
  \[ [p]\le [q]:= \ex {p'}P p \le p' \conjt p' \theta q. \]
  Firstly, this relation is well-defined: suppose $p\theta r$ and $q \theta s$, and we have $p \le p'$ that $p' \theta q$. Then as shown below, we can find $r \le r'$ that $r'\theta p'$ and thus $r' \theta s$,
    \[\begin{tikzcd}
    	{r'} & {p'} & q & s \\
    	r & p
    	\arrow["\theta", dashed, tail reversed, from=1-1, to=1-2]
    	\arrow["\theta", tail reversed, from=1-2, to=1-3]
    	\arrow["\theta", tail reversed, from=1-3, to=1-4]
    	\arrow[dashed, from=2-1, to=1-1]
    	\arrow["\theta"', tail reversed, from=2-1, to=2-2]
    	\arrow[from=2-2, to=1-2]
    \end{tikzcd}\]
  Hence, the definition does not depend on the representative we choose.

  Furthermore, we show the relation is a partial order on equivalence classes. Reflexivity is trivial. Transitivity holds by the following diagram,
    \[\begin{tikzcd}
    	& {p''} & {q'} & r \\
    	p & {p'} & q
    	\arrow["\theta", dashed, tail reversed, from=1-2, to=1-3]
    	\arrow["\theta", tail reversed, from=1-3, to=1-4]
    	\arrow[from=2-1, to=2-2]
    	\arrow[dashed, from=2-2, to=1-2]
    	\arrow["\theta"', tail reversed, from=2-2, to=2-3]
    	\arrow[from=2-3, to=1-3]
    \end{tikzcd}\]
  For anti-symmetry, suppose $[p] \le [q]$ and $[q] \le [p]$. Then we can construct the following diagram,
    \[\begin{tikzcd}
    	& \vdots && \vdots \\
    	& {p_1} & {p_0} & p \\
    	p & {q_0} & q
    	\arrow[dashed, from=2-2, to=1-2]
    	\arrow["\theta", dashed, tail reversed, from=2-2, to=2-3]
    	\arrow["\theta", tail reversed, from=2-3, to=2-4]
    	\arrow[from=2-4, to=1-4]
    	\arrow[from=3-1, to=3-2]
    	\arrow[dashed, from=3-2, to=2-2]
    	\arrow["\theta"', tail reversed, from=3-2, to=3-3]
    	\arrow[from=3-3, to=2-3]
    \end{tikzcd}\]
  In other words, we have a chain 
  \[ p \le q_0 \le p_1 \le q_1 \le p_2 \le \cdots \]
  and each $p_n \theta p$ and $q_n \theta q$. Since $P$ is \emph{finite}, this chain must stabilise at some point, say $p_n = q_{n+1}$. Then we have $p \theta p_n = q_{n+1} \theta q$, which means $[p]=[q]$. Thus, this is a partial order. Finally, it is then easy to see that the map $\pi : P \surj P/\theta$ is the coequaliser of $f_0,f_1$ in $\Pos$, hence also in $\Kp$.
\end{proof}

However, $\Kp$ does not admit all finite limits. It has a terminal object $1$, since every monotone map into $1$ is automatically open, but $\Kp$ does not in general admit pullbacks, or even products. However, the pullback operation in $\Pos\fn$ does land in $\Kp$:

\begin{lemma}\label{monoidalpullback}
    Let $f : P \to Q$ be an open map between posets. For any $g : R \to Q$, let $R \otimes_Q P$ be the pullback in $\Pos$. In this case, $h$ is open as well.
    \[
    \begin{tikzcd}
        R \otimes_Q P \ar[r, "u"] \ar[d, "h"'] & P \ar[d, "f"] \\
        R \ar[r, "g"'] & Q
    \end{tikzcd}
    \]
\end{lemma}
\begin{proof}
    Take any pair $(r,p) \in R \otimes_Q P$. Suppose $h(r,p) = r \le r'$ in $R$. Then we have $fp = gr \le gr'$ in $Q$. By openness of $f$, there exists $p'\in P$ that $fp' = gr'$. By construction, $(r',p') \in R\otimes_Q P$, $(r,p) \le (r',p')$, and $h(r',p') = r'$. Hence, $h$ is open.
\end{proof}

\cref{monoidalpullback} implies that there is a symmetric semi-cartesian monoidal structure $\otimes$ on $\Kp$, which is simply given by the cartesian product of posets. In fact, $\otimes_Q$ will also provide a monoidal structure on the slice category $\Kp/Q$, and we call this the \emph{monoidal pullback}. This operation will play an important role in defining the topology on $\Kp$.


With these elementary results in place, we now recall the Grothendieck topology $J$ on $\Kp$ introduced in~\cite[Sec. 4.2]{ghilardi2013sheaves}:

\begin{definition}[The Grothendieck topology on Kripke frames]\label{topologyonKp}
  For $P\in\Kp$, a sieve $S$ on $P$ is $J$-covering iff it is jointly surjective.
\end{definition}

Given the monoidal pullback, seeing this is a well-defined topology is not so hard:

\begin{lemma}\label{topology}
  The set of sieves $J$ specified in~\cref{topologyonKp} is a Grothendieck topology on $\Kp$.
\end{lemma}
\begin{proof}
  Maximal sieves contain the identity, hence are covering sieves. Suppose we have a covering sieve $S$ on $P$. For any $g : R \to P$ in $\Kp$ and $f : Q \to P$ in $S$, consider the monoidal pullback in $\Kp$,
  \[
  \begin{tikzcd}
    R \otimes_P Q \ar[d, "h"'] \ar[r, "g"] & Q \ar[d, "f"] \\
    R \ar[r, "g"'] & P
  \end{tikzcd}
  \]
  This shows that $h\in g^*S$. Since the monoidal pullback is simply the pullback of sets, $g^*S$ will also be jointly surjective. Finally, jointly surjective families are evidently composable.
\end{proof}

$J$ is in fact the \emph{canonical topology} on $\Kp$, as observed in~\cite[Prop. 4.7]{ghilardi2013sheaves}.

\begin{proposition}\label{canonicaltopology}
  The following are equivalent for a sieve $S$ on $P$ in $\Kp$:
  \begin{enumerate}
    \item $S$ is $J$-covering;
    \item $S$ is effectively epimorphic;
    \item $S$ is universally effectively epimorphic.
  \end{enumerate}
  In particular, $J$ is the canonical topology on $\Kp$.
\end{proposition}
\begin{proof}
  (1) $\nt$ (2): Suppose $S$ is a $J$-covering. We show this is a colimit diagram. Suppose for any $f : Q \to P$ in $S$ we have a map $g_f : Q \to R$, such that for any $h : Q' \to Q$ we have $g_f h = g_{fh}$. We construct a map $g : P \to R$ as follows. For any $p\in P$, choose $f : Q \to P$ with $f(q) = p$. Define $g(p)$ to be $g_f(q)$. To show this is well-defined, suppose we have another $f' : Q' \to P$ with $f'(q') = p$. Consider the following diagram,
  \[\begin{tikzcd}[row sep = small]
    & {Q'} \\
    {Q\otimes_PQ'} \arrow[ur, "{\pi'}"] \arrow[dr, "\pi"'] && P & R \\
    & Q
    \arrow["{f'}"{description}, from=1-2, to=2-3]
    \arrow["{g_{f'}}", curve={height=-12pt}, from=1-2, to=2-4]
    \arrow["g"{description}, dashed, from=2-3, to=2-4]
    \arrow["f"{description}, from=3-2, to=2-3]
    \arrow["{g_f}"', curve={height=12pt}, from=3-2, to=2-4]
  \end{tikzcd}\]
  By assumption, we have 
  \[ g_f(q) = g_f(\pi(q,q')) = g_{f\pi}(q,q') = g_{f'\pi'}(q,q') = g_{f'}(q'). \]
  Hence, the map $g$ is well-defined, and uniqueness follows from the definition.

  (2) $\nt$ (1): Suppose now $S$ is a colimit cone on $P$. Now if $S$ is not jointly surjective, evidently when we have a cocone $S \to R$ for some $R\in\Kp$, we can send any $p\in P$ outside the image of $S$ to an arbitrary element in $R$ below the image of $S \to R$, a contradiction. Hence, $S$ is jointly surjective.

  Combining the above, if $S$ is effectively epic, then it is a $J$-covering, but we have shown in \Cref{topology} that $J$ is stable under pullbacks, thus $S$ is universally effectively epic. Hence (1) $\eff$ (2) $\eff$ (3).
\end{proof}

\begin{remark}[Sheaves and left exact functors]\label{leftexactfuncaresheaf}
    Since each poset $P\in\Kp$ is finite, a sieve $S$ on $P$ is a $J$-covering iff it contains a \emph{finite} family $\set{Q_i \to P}_{i\in n}$ which is already jointly surjective. By the proof of~\cref{canonicaltopology}, $P$ is already the colimit of the family of spans $\set{Q_i \leftarrow Q_i \otimes_{P} Q_j \rightarrow Q_j}_{i,j\in n}$. This way, if a functor $F : \Kp\op \to \Set$ preserves finite limits, $FP$ will be the limit of the corresponding cospans $\lt\set{FQ_i \to F(Q_i \otimes_P Q_j) \leftarrow FQ_j}_{i,j\in n}$, hence also the limit $\lt_{Q\to P\in S}FQ$ of the image of sieve under $F$. This shows left exact functors are sheaves on $\Kp$. However, we will see that not all sheaves are of this form. In fact,~\cref{sec:nerveonK} will show that these sheaves exactly correspond to profinite Heyting algebras (cf.~\cref{filteredIKp}).
\end{remark}

Henceforth, we will denote the sheaf topos $\sh(\Kp,J)$ as $\mc K$, and refer to it as the \emph{K-topos}. Since $J$ is the canonical topology, the Yoneda embedding factors through $\mc K$ and we have a fully faithful embedding $\yon : \Kp \hook \mc K$. However, we often omit explicitly mentioning $\yon$ and identify $P\in\Kp$ with the representable functor in $\mc K$.

\section{Some topos-theoretic properties of the K-topos}\label{sec:property-K}

In this section we establish some topos-theoretic properties of the K-topos $\mc K$. Many of them will have application to profinite Heyting algebras and uniform interpolation later, when we understand better the relationship between these subjects and the K-topos.

\subsection{Supercompact generation}

$\mc K$ is \emph{supercompactly generated} in the sense of~\cite{RN900}. An object $e$ in a topos $\mc E$ is \emph{supercompact} if every jointly epic family of maps $\set{e_i \to e}_{i\in I}$ contains an actual epimorphism. And $\mc E$ is supercompactly generated if it has a generating family of supercompact objects. 

To show $\mc K$ is supercompactly generated, we consider another site presentation of $\mc K$, which is also considered in~\cite[Sec. 4.2]{ghilardi2013sheaves}. Let $\Kpp$ be the full subcategory of $\Kp$ consisting of \emph{rooted} finite posets, i.e. those posets containing a minimum. We may look at the site $(\Kpp,J)$ where $J$ is the topology on $\Kp$ restricted to $\Kpp$. These two sites generate the same sheaf topos:

\begin{lemma}\label{comparisonK}
  The inclusion $\Kpp \hook \Kp$ induces equivalent sheaf topoi,
  \[ \sh(\Kp,J) \simeq \sh(\Kpp,J). \]
\end{lemma}
\begin{proof}
  By the comparison lemma (cf.~\cite[C2.2]{johnstone2002sketches}), it suffices to show every object in $\Kp$ admits a covering family of objects from $\Kpp$. For any finite poset $P$, the family of open embeddings $\set{\dv p \hook P}_{p\in P}$ is a $J$-covering, and the domains $\dv p$ are all rooted.
\end{proof}

\begin{proposition}\label{supercompact}
  $\mc K$ is supercompactly generated.
\end{proposition}
\begin{proof}
  Firstly, notice that if $P$ is a rooted poset, then a sieve $S$ is a $J$-covering sieve on $P$ iff it contains a surjection: the sieve being jointly epic implies the root of $P$ lies in the image of some $f : Q \to P$, and $f$ being open implies it must be a surjection. Thus, the site $(\Kpp,J)$ is a principal site in the sense of \cite[Def. 2.1.6]{RN900}, and the representables from $\Kpp$ will all be supercompact in $\sh(\Kpp,J) \simeq \mc K$, which form a generating family.
\end{proof}

\begin{corollary}\label{unionofsub}
  If we view $\mc K$ as a subtopos of the presheaf topos $[\Kpp\op,\Set]$, then the inclusion $\mc K \hook [\Kpp\op,\Set]$ preserves arbitrary unions of subobjects.
\end{corollary}
\begin{proof}
  This follows from~\cite[Cor. 2.1.8]{RN900}.
\end{proof}

\begin{corollary}\label{subKcompdistr}
  $\mc K$ is completely distributive in the sense of~\cite{BorceuxFrancis2006Inci}, i.e. the subobject lattice of any sheaf in $\mc K$ is a completely distributive lattice.
\end{corollary}
\begin{proof}
  This is simply due to the fact that both intersections and unions of subobjects are computed pointwise over $\Kpp$, thus they are completely distributive.
\end{proof}

\begin{remark}[Subobject lattices in the K-topos]
    In fact one can say more about subobject lattices in $\mc K$. Since $\mc K$ is supercompactly generated, subobject lattices in $\mc K$ will also be supercompactly generated as frames. By~\cite[Thm 4.2]{caramello2011topos}, they will be lattices of upward closed subsets of certain posets. \cref{sec:nerveonK} will give a characterisation of which posets arise this way.
\end{remark}

\subsection{Subobject classifier}

The subobject classifier $\Omega$ in $\mc K$ will be crucial for applications to uniform interpolation which will be discussed in~\cref{sec:internaluniform}. Thus, here we provide an explicit description of $\Omega$. 

\begin{lemma}\label{subobject}
  The subobject classifier $\Omega$ of $\mc K$ as a sheaf on $\Kp$ is given by 
  \[ \Omega(P) \cong \mc UP. \]
\end{lemma}
\begin{proof}
  It is well-known that $\Omega(P)$ can be described as the set of $J$-closed sieves on $P$ (cf. the proof of~\cite[Lem. A.2.1.10]{johnstone2002sketches}). We show that every $J$-closed sieve $S$ on $P$ is generated by some open inclusion $Q \hook P$ in $\Kp$. Given such an $S$, consider the union of images of maps in $S$ and denote it as the inclusion $i : Q \hook P$. By construction, $i^*(S)$ is jointly surjective on $Q$, thus is a $J$-covering. By $J$-closedness, $i \in S$, thus since every map in $S$ factors through $i$, $S$ is generated by $i$. This shows that $J$-closed sieves on $P$ can be identified with upward closed subsets of $P$, hence $\Omega(P) \cong \mc UP$.
\end{proof}

As an easy consequence, we observe immediately that $\mc K$ is 2-valued. Let $\Sigma = \set{0<1}$ be the Sierpi\'nski poset.

\begin{corollary}\label{twovalued}
  $\mc K$ is two-valued, i.e. $\Omega$ has exactly two global points.
\end{corollary}
\begin{proof}
  We can directly compute the global sections of $\Omega$ via~\cref{subobject},
  \[ \mc K(1,\Omega) \cong \Omega(1) \cong \mc U1 \cong \Sigma. \]
  Thus it is two-valued.
\end{proof}

Equivalently, this is saying that $\mc K$ is a \emph{hyperconnected topos} (cf.~\cite[A4.6]{johnstone2002sketches}). We will meet different presentations of the subobject classifier in~\cref{sec:nerveonK}. In~\cref{sec:internaluniform} we will establish more properties about the subobject classifier, and see how it connects to properties of uniform interpolation.

\subsection{The K-topos as a gros topos}

In the terminology of~\cite{sufficientcohesion}, the K-topos (over $\Set$) is \emph{pre-cohesive} with \emph{sufficient cohesion}.

\begin{proposition}\label{Kcohesive}
  The global section functor $\Gamma : \mc K \to \Set$ makes $\mc K$ pre-cohesive over $\Set$ in the sense of~\cite[Def. 2.1]{sufficientcohesion}, or a punctually locally connected and local topos in the sense of~\cite{johnstone2011remarks}. Concretely, this means that there is an adjoint quadruple between $\mc K$ and $\Set$, 
  \[\begin{tikzcd}
    {\mc K} & \Set
    \arrow[""{name=0, anchor=center, inner sep=0}, "\Gamma"{description}, curve={height=8pt}, from=1-1, to=1-2]
    \arrow[""{name=1, anchor=center, inner sep=0}, "\Pi", curve={height=-21pt}, from=1-1, to=1-2]
    \arrow[""{name=2, anchor=center, inner sep=0}, "\nabla", curve={height=-21pt}, from=1-2, to=1-1]
    \arrow[""{name=3, anchor=center, inner sep=0}, "\Delta"{description}, curve={height=8pt}, from=1-2, to=1-1]
    \arrow["\dashv"{anchor=center, rotate=-90}, draw=none, from=1, to=3]
    \arrow["\dashv"{anchor=center, rotate=-90}, draw=none, from=0, to=2]
    \arrow["\dashv"{anchor=center, rotate=-90}, draw=none, from=3, to=0]
  \end{tikzcd}\]
  satisfying the following conditions:
  \begin{enumerate}
      \item The right adjoint $\nabla$ of $\Gamma$, hence also the constant sheaf functor $\Delta$, are fully faithful;
      \item The connected component functor $\Pi$ preserves finite products;
      \item The canonical natural transformation $\Gamma \to \Pi$ is pointwise epic.
  \end{enumerate}
\end{proposition}
\begin{proof}
  We use the site characterisation provided in~\cite{johnstone2011remarks}. Notice that $(\Kpp,J)$ has a terminal object $1$. For any $J$-covering sieve on $P$ in $\Kpp$, by definition it is inhabited since it contains some open surjection $Q \surj P$; it is then also connected as a subcategory of $\Kpp/P$ due to the existence of the monoidal pullback operator (cf.~\cref{monoidalpullback}). Furthermore, any $P\in\Kpp$ has a global point $1 \to P$ corresponding to a maximal element in $P$. Thus, $(\Kpp,J)$ is a connected, punctually locally connected, and local site in the sense of~\cite{johnstone2011remarks}, thus $\mc K$ has these properties.
\end{proof}

\begin{corollary}\label{nabladoublenegation}
    $\nabla : \Set \to \mc K$ is the inclusion of $\neg\neg$-sheaves on $\mc K$.
\end{corollary}
\begin{proof}
    Notice that from~\cref{Kcohesive}, $\nabla$ is fully faithful and its left adjoint $\Gamma$ preserves finite, in fact all, limits. Thus, $\nabla : \Set \to \mc K$ identifies $\Set$ as a \emph{subtopos} of $\mc K$. By~\cite[Cor. 4.5]{lawvere_internal_nodate} it coincides with $\mc K_{\neg\neg}$, the subtopos for the double negation topology. We will come back to this in~\cref{thm:subtoposofK}.
\end{proof}


$\mc K$ furthermore satisfies sufficient cohesion in the sense of~\cite{lawvere2007axiomatic}; also see~\cite{sufficientcohesion}:

\begin{corollary}\label{sufficientcohesion}
  $\mc K$ satisfies sufficient cohesion, in the sense that every object in $\mc K$ embeds into a contractible object, where $X$ is contractible iff for any $A$, $\Pi(X^A) \cong 1$.
\end{corollary}
\begin{proof}
  From~\cite{sufficientcohesion,lawvere2007axiomatic} it suffices to show the subobject classifier $\Omega$ is connected, i.e.\ $\Pi\Omega \cong 1$. This can either be seen from the alternative description of the subobject classifier in~\cref{subobjectasanerve}, or it also follows easily from~\cite[Prop. 2.14]{menni2014sufficient}. For completeness, we also provide a straightforward proof. This holds because we have
  \[ \Pi \Omega \cong \ct_{P\in\Kpp}\Omega(P) \cong \ct_{P\in\Kpp}\mc UP \cong 1. \]
  In the colimit, any two elements $a\in\mc UP$ and $b\in\mc UQ$ are identified because there are projections $\pi_P,\pi_Q : P \otimes Q \to P,Q$, and we have the element $a\otimes b$ in $\mc U(P \otimes Q) \cong \mc UP + \mc UQ$, where here $+$ denotes the coproduct of distributive lattices and it is computed by a tensor product (cf.~\cite[Ch. III.2]{joyal1984extension}).
\end{proof}

\cite{menni2019every} has introduced a notion of a \emph{perfect} topos, which is in some sense a complementary notion of \emph{scattered} topos introduced in~\cite{ESAKIA200097}. By~\cite[Cor. 7.3]{menni2019every}, a topos is perfect if its only open Boolean subtopos is the degenerate one.

\begin{corollary}
    $\mc K$ is perfect.
\end{corollary}
\begin{proof}
    By~\cite[Cor. 7.4]{menni2019every}, $\mc K$ is perfect if it is non-Boolean, viz.\ $\Omega \not\cong 1 + 1$ in $\mc K$. This can be directly seen as a consequence of~\cref{sufficientcohesion}, since $\Pi\Omega \cong 1$ while $\Pi(1+1) \cong 2$.
\end{proof}

\begin{remark}
    The fact that $\mc K$ is non-Boolean can also be seen from the concrete description of $\Omega$ in~\cref{subobjectasanerve}. Later in~\cref{thm:subtoposofK}, we in fact will provide a complete classification of subtopoi of $\mc K$, thus proving a much stronger result than $\mc K$ being perfect.
\end{remark}

\section{Nerve construction over the K-topos}\label{sec:nerveonK}

The goal of this section is to explain the relation between the sheaf representation of Heyting algebras as sheaves in the K-topos and their profinite completion. As mentioned in~\cref{sec:intro}, we will follow a categorical approach by describing the \emph{nerve} construction over the K-topos. 

The K-topos $\mc K \simeq \sh(\Kp,J)$ as a sheaf topos over the canonical topology on $\Kp$ is in particular a cocompletion of $\Kp$, though not the free one. Via left Kan extension, suitable functors $\Kp \to \mc E$ will induce nerve adjunctions between $\mc E$ and $\mc K$.

\begin{definition}[Representably continuous functors]
  A functor $F : \Kp \to \mc E$ is \emph{representably $J$-continuous} if for any $e\in\mc E$, $\mc E(F-,e) : \Kp\op \to \Set$ is a $J$-sheaf on $\Kp$. 
\end{definition}

\begin{remark}[Right exact functors are representably continuous]\label{rightexactrep}
    Recall from~\cref{leftexactfuncaresheaf} that left exact functors from $\Kp\op$ to $\Set$ are $J$-sheaves. By Yoneda, a functor $F : \Kp \to \mc E$ preserves finite colimits iff $\mc E(F-,e)$ preserves finite limits for all $e\in\mc E$. In particular, right exact functors from $\Kp$ to $\mc E$ are representably $J$-continuous.
\end{remark}

\begin{proposition}\label{nerve}
  Let $F : \Kp \to \mc E$ be a representably $J$-continuous functor into a cocomplete category. Then there is an induced adjunction
    \[\begin{tikzcd}
    	\Kp & {\mc E} \\
    	{\mc K}
    	\arrow["F", from=1-1, to=1-2]
    	\arrow["\yon"', hook, from=1-1, to=2-1]
    	\arrow[""{name=0, anchor=center, inner sep=0}, "N", curve={height=-8pt}, from=1-2, to=2-1]
    	\arrow[""{name=1, anchor=center, inner sep=0}, "L", curve={height=-8pt}, from=2-1, to=1-2]
    	\arrow["\dashv"{anchor=center, rotate=-40}, draw=none, from=1, to=0]
    \end{tikzcd}\]
  where $L$ is obtained by left Kan extension $L \cong \lan_{\yon}F$, and $N$ is the nerve $Ne \cong \mc E(F-,e)$.
\end{proposition}
\begin{proof}
  By assumption $N$ is well-defined since each $Ne$ is a $J$-sheaf for any $e\in\mc E$. Since $L \cong \lan_{\yon} F$ is a left Kan extension into a cocomplete category, it is computed pointwise,
  \[ LX \cong \ct_{x\in X(P),P\in\Kp}FP. \]
  This way, for any $X\in\mc K$ and $e\in\mc E$ we have the following natural isomorphisms,
  \begin{align*}
    \mc E(LX,e) 
    &\cong \lt_{x\in X(P)}\mc E(FP,e) \\ 
    &\cong \lt_{x\in X(P)}\mc K(P,Ne) \\
    &\cong \mc K(X,Ne)
  \end{align*}
  The first and third isomorphisms hold by the colimit formula of $LX$, and the second isomorphism holds by Yoneda. Thus, we have an adjunction $L \dashv N$. 
\end{proof}

\begin{remark}
    The nerve adjunction was first introduced in~\cite{kan1958functors}, where an arbitrary functor $F : \mc C \to \mc E$ from a small category into a cocomplete category induces an adjunction between $\psh(\mc C)$ and $\mc E$. In modern literature, the left adjoint is often called the \emph{geometric realisation}, and the right adjoint is often called the \emph{nerve}. \cref{nerve} is a generalisation of this construction for \emph{sheaf}, instead of presheaf, categories, thus we borrow the same terminology.
\end{remark}

Notice that since $L$ is the left Kan extension along the fully faithful embedding $\yon$, we in fact have $L\yon \cong F$. The remainder of this section proceeds to discuss various important examples of the nerve construction over $\mc K$.

\subsection{Nerve for the dual category of profinite Heyting algebras}

Let $\Pro(\HA\fn)$ be the category of profinite Heyting algebras, viz. the pro-completion of the category of finite Heyting algebras $\HA\fn$. By~\cite[Ch. VI.2.10]{johnstone1982stone}, $\Pro(\HA\fn)$ also has a concrete description as the category of internal Heyting algebras in the category of Stone spaces. It has been shown in~\cite{RN921} that profinite Heyting algebras can also be characterised via pure lattice-theoretic properties: a Heyting algebra is profinite iff (i) it admits an injection into a product of finite Heyting algebras; (ii) it is complete; and (iii) every element is a join of completely join-irreducible elements. In particular, any Stone topology on a Heyting algebra making the Heyting operations continuous will necessarily be \emph{unique}, and the existence of such a topology implies and is implied by (i--iii) above.

As mentioned in~\cref{sec:intro}, the dual category $\Pro(\HA\fn)\op$ is a reflective subcategory of $\mc K$, and we describe this reflective adjunction as a special case of the nerve adjunction given in~\cref{nerve}. We first mention a more concrete description of the dual category of profinite Heyting algebras. Notice that by construction we have the following equivalences, 
\[ \Pro(\HA\fn)\op \simeq \Ind(\HA\fn\op) \simeq \Ind(\Kp). \]
In fact, it is not hard to directly compute the ind-completion of $\Kp$. \cite{RN921} provides a concrete description of $\Ind(\Kp)$ as the category of \emph{image finite posets} and open maps between them. A poset $X$ is image finite if for all $x\in X$, the set $\dv x$ is finite. Open maps between image finite posets share the same definition as for finite posets in~\cref{sec:K-topos}, as the notion of open map only involves the upset of an element. We will write $\IKp$ for the category of image finite posets with open maps between them. In this concrete description, the duality between profinite Heyting algebras and image finite posets is given by
\[ \mc U : \IKp \simeq \Pro(\HA\fn)\op : \mc J. \]
Here $\mc U$ takes an image finite poset to its lattice of upward closed sets, while $\mc J$ takes a profinite Heyting algebra to the dual poset of completely join-irreducible elements.

Since $\Kp$ has finite colimits, the ind-completion $\IKp \simeq \Ind(\Kp)$ is locally finitely presentable, thus in particular cocomplete, and the inclusion $\Kp \hook \IKp$ preserves finite colimits. By~\cref{rightexactrep} and~\cref{nerve}, we get an adjunction
\[\begin{tikzcd}
	{\mc K} & \IKp
	\arrow[""{name=0, anchor=center, inner sep=0}, "{\geo-}", curve={height=-15pt}, from=1-1, to=1-2]
	\arrow[""{name=1, anchor=center, inner sep=0}, "N", from=1-2, to=1-1]
	\arrow["\dashv"{anchor=center, rotate=-90}, draw=none, from=0, to=1]
\end{tikzcd}\]
As promised, we show this adjunction is a \emph{reflection}, or equivalently $N$ is fully faithful:

\begin{proposition}\label{IKpreflective}
  The nerve adjunction between $\Pro(\HA\fn)\op \simeq \IKp$ and $\mc K$ is a reflection. 
\end{proposition}
\begin{proof}
  To show $N : \IKp \to \mc K$ is fully faithful, take $X,Y\in\IKp$. By construction,
  \[ \mc K(N(X),N(Y)) \cong \int_{P\in\Kp}\Set(\IKp(P,X),\IKp(P,Y)). \]
  On the other hand, since $\Kp$ has finite colimits, we also have the following equivalences
  \[ \IKp \simeq \Ind(\Kp) \simeq \Lex(\Kp\op,\Set), \]
  where as a full subcategory of $[\Kp\op,\Set]$, an object $X \in \IKp$ is exactly identified with the functor $P \mapsto \IKp(P,X)$. This implies that
  \[ \int_{P\in\Kp}\Set(\IKp(P,X),\IKp(P,Y)) \cong \Lex(\Kp\op,\Set)(X,Y) \cong \IKp(X,Y). \qedhere \]
\end{proof}

\begin{remark}[Image finite posets as left exact sheaves]\label{filteredIKp}
  By the proof of~\cref{IKpreflective}, the inclusion $\IKp \hook \mc K$ is the unique factorisation of two reflective subcategories of the presheaf category $[\Kp\op,\Set]$,
  \[
  \begin{tikzcd}
    \IKp \ar[r, hook] \ar[dr, hook] & \mc K \ar[d, hook] \\ 
    & {[\Kp\op,\Set]}
  \end{tikzcd}
  \]
  where $\IKp,\mc K$ correspond to left exact and $J$-continuous functors out of $\Kp\op$, respectively. From now on we will identify $\IKp$ as a full subcategory of $\mc K$. The following are equivalent for a sheaf $F$ in $\mc K$:
  \begin{itemize}
    \item $F \cong NX$ for an essentially unique $X\in\IKp$;
    \item $F : \Kp\op \to \Set$ preserves finite limits; 
    \item The category of elements $\elem F$ is cofiltered.
  \end{itemize}
\end{remark}

\begin{remark}[The nerve of profinite Heyting algebras]\label{nerveofProHA}
    Via the dual equivalence 
    \[ \mc U : \IKp \simeq \Pro(\HA\fn)\op : \mc J, \] 
    for any profinite Heyting algebra $A$ its nerve in $\mc K$ is given by
    \[ N\mc JA(P) \cong \IKp(P,\mc JA) \simeq \Pro(\HA\fn)(A,\mc UP). \]
\end{remark}

\begin{example}[Subobject classifier as a profinite Heyting algebra or an image finite poset]\label{subobjectasanerve}
  Let $H[\ms x]$ be the free Heyting algebra on one generator. Concretely, it is the \emph{Rieger-Nishimura lattice} (cf.~\cite{nishimura1960formulas}). It is well-known that $H[\ms x]$ is already profinite; in fact it is isomorphic to its profinite completion (cf.~\cite{RN921}). By~\cref{nerveofProHA}, its nerve is given by
  \[ N\mc JH[\ms x](P) \cong \Pro(\HA\fn)(H[\ms x],\mc UP) \cong \mc UP \cong \Omega(P). \]
  The final isomorphism is calculated in~\cref{subobject}, where $\Omega$ is the subobject classifier in $\mc K$. This shows that $\Omega\in\mc K$ belongs to the reflective subcategory $\Pro(\HA\fn)\op$. Under the equivalence $\Pro(\HA\fn)\op \simeq \IKp$, $\Omega$ can be identified as the \emph{Rieger-Nishimura ladder} below,
  \[\begin{tikzcd}
    \bullet & \bullet \\
    \bullet & \bullet \\
    \bullet & \bullet \\
    \vdots & \vdots
    \arrow[from=2-1, to=1-1]
    \arrow[from=2-1, to=1-2]
    \arrow[from=2-2, to=1-2]
    \arrow[from=3-1, to=2-1]
    \arrow[from=3-1, to=2-2]
    \arrow[from=3-2, to=1-1]
    \arrow[from=3-2, to=2-2]
    \arrow[from=4-1, to=3-1]
    \arrow[from=4-1, to=3-2]
    \arrow[from=4-2, to=2-1]
    \arrow[from=4-2, to=3-2]
  \end{tikzcd}\]
\end{example}

\begin{corollary}\label{closedundersubobjects}
  $\IKp \hook \mc K$ is closed under subobjects in $\mc K$, i.e. for any $X\in\IKp$, if $S\inj X$ is a subobject in $\mc K$, then $S\in\IKp$ as well.
\end{corollary}
\begin{proof}
  Recall that $\IKp\hook\mc K$ is a full subcategory closed under all limits. Since $\Omega$ and $1$ belong to $\IKp$ and any subobject of $X\in\IKp$ is a pullback of $\top : 1 \to \Omega$ along a map $X \to \Omega$, it follows that the subobject also belongs to $\IKp$.
\end{proof}

\begin{proposition}\label{leftadjointofnerveIKp}
    Under the equivalence $\mc U : \IKp \simeq \Pro(\HA\fn)\op : \mc J$, the left adjoint to the inclusion $\Pro(\HA\fn)\op \hook \mc K$ is given by taking subobjects of sheaves,
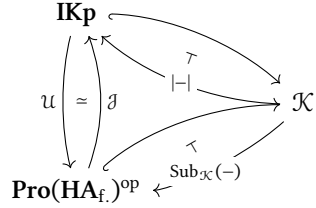
    \[\begin{tikzcd}[column sep = 8ex]
    	& \IKp \\
    	{\mc K} \\
    	& {\Pro(\HA\fn)\op}
    	\arrow[""{name=0, anchor=center, inner sep=0}, hook', curve={height=-10pt}, from=1-2, to=2-1]
    	\arrow["\simeq"{description}, tail reversed, from=1-2, to=3-2]
    	\arrow[""{name=1, anchor=center, inner sep=0}, "{\geo-}"{description}, curve={height=-10pt}, from=2-1, to=1-2]
    	\arrow[""{name=2, anchor=center, inner sep=0}, "{\sub_{\mc K}(-)}"{description}, curve={height=-10pt}, from=2-1, to=3-2]
    	\arrow[""{name=3, anchor=center, inner sep=0}, hook', curve={height=-10pt}, from=3-2, to=2-1]
    	\arrow["\dashv"{anchor=center, rotate=-60}, draw=none, from=1, to=0]
    	\arrow["\dashv"{anchor=center, rotate=-120}, draw=none, from=2, to=3]
    \end{tikzcd}\]
\end{proposition}
\begin{proof}
    By construction, it suffices to show the left adjoint to the nerve $\Pro(\HA\fn)\op \hook\mc K$ is given by $\sub_{\mc K}(-)$. Since $H[\ms x]$ is profinite, it is also the free profinite Heyting algebra on one generator. Let $L : \mc K \to \Pro(\HA\fn)\op$ be the left adjoint of the nerve. For any $X\in\mc K$,
    \begin{align*}
        LX &\cong \Pro(\HA\fn)(H[\ms x],LX) \\
        &\cong \Pro(\HA\fn)\op(LX,H[\ms x]) \\
        &\cong \mc K(X,\Omega) \\
        &\cong \sub_{\mc K}(X).
    \end{align*}
    The first isomorphism holds by the universal property of $H[\ms x]$; the second isomorphism is formal; the third isomorphism holds by the nerve adjunction and the computation in~\cref{subobjectasanerve} that the nerve of $H[\ms x]$ is $\Omega$. This shows $L$ is indeed given by $\sub_{\mc K}(-)$.
\end{proof}

\begin{corollary}\label{profiniteHAsameassubK}
    A lattice $A$ is a profinite Heyting algebra iff it is isomorphic to the subobject lattice of a sheaf in $\mc K$.
\end{corollary}

\begin{remark}[Profinite Heyting algebras are bi-Heyting algebras]
    \cref{profiniteHAsameassubK} together with~\cref{subKcompdistr} implies that all profinite Heyting algebras are \emph{completely distributive}, thus in particular they are also bi-Heyting algebras. Of course this is also a direct consequence of the dual equivalence $\mc U : \IKp \simeq \Pro(\HA\fn)\op$, as observed in~\cite{RN921}.
\end{remark}

Finally, we observe the following characterisation of when a general nerve construction over $\mc K$ factors through the reflective subcategory $\IKp$. This will be the key to describing the connection between profinite completion and sheaf representation of Heyting algebras in~\cref{subsec:nerveHA}. 

\begin{proposition}\label{representationinIKp}
  Let $F : \Kp \to \mc E$ be representably $J$-continuous. The nerve adjunction $L \dashv N : \mc K \to \mc E$ factors through $\IKp$ as follows iff $F$ preserves finite colimits,
    \[\begin{tikzcd}
    	{\mc K} & \IKp & {\mc E}
    	\arrow[""{name=0, anchor=center, inner sep=0}, curve={height=-18pt}, from=1-1, to=1-2]
    	\arrow[""{name=1, anchor=center, inner sep=0}, hook', from=1-2, to=1-1]
    	\arrow[""{name=2, anchor=center, inner sep=0}, curve={height=-18pt}, from=1-2, to=1-3]
    	\arrow[""{name=3, anchor=center, inner sep=0}, from=1-3, to=1-2]
    	\arrow["\dashv"{anchor=center, rotate=-90}, draw=none, from=0, to=1]
    	\arrow["\dashv"{anchor=center, rotate=-90}, draw=none, from=2, to=3]
    \end{tikzcd}\]
\end{proposition}
\begin{proof}
  This is a direct consequence of~\cref{rightexactrep}, the characterisation of $\IKp$ as a full subcategory in~\cref{filteredIKp}, and the fact that left Kan extensions compose.
\end{proof}

\subsection{Sheaf representation of Heyting algebras as a nerve}\label{subsec:nerveHA}

As mentioned in~\cref{sec:intro},~\cite{ghilardi1995sheaf} has constructed a sheaf representation functor of Heyting algebras,
\[ \kspec : \HA\op \to \mc K, \]
which takes any Heyting algebra $A$ to the following sheaf
\[ \kspec A(P) = \HA(A,\mc UP). \]
In other words, $\kspec A(P)$ is exactly the set of Kripke models of $A$ over the finite Kripke frame $P$. We first notice that this is a special case of the nerve construction:

\begin{lemma}
    The functor $\kspec : \HA\op \to \mc K$ described above is the nerve induced by the inclusion $\mc U : \Kp \simeq \HA\fn\op \hook \HA\op$.
\end{lemma}
\begin{proof}
    Notice that $\HA$ is complete, thus $\HA\op$ is cocomplete. Furthermore, the inclusion $\HA\fn \hook \HA$ preserves finite limits, thus $\mc U : \Kp \hook \HA\op$ preserves finite colimits. By construction, its nerve takes $P\in\Kp$ to the set
    \[ \HA\op(\mc UP,A) \cong \HA(A,\mc UP) \cong \kspec A(P). \]
    This shows $\kspec$ is indeed the nerve functor for the inclusion $\mc U : \Kp \hook \HA\op$.
\end{proof}

Since $\mc U : \Kp \hook \HA\op$ preserves finite colimits,~\cref{representationinIKp} shows that this nerve functor factors through $\IKp \simeq \Pro(\HA\fn)\op$, and the following is straightforward:

\begin{theorem}\label{sheafrepasprocompletion}
    The sheaf representation functor $\kspec : \HA\op \to \mc K$ factors through the reflective subcategory $\IKp$. Under the equivalence $\IKp \simeq \Pro(\HA\fn)\op$, this is exactly the dual of the profinite completion functor $\pf - : \HA \to \Pro(\HA\fn)$,
    \[\begin{tikzcd}
    	\IKp & \\
    	& {\HA\op} \\
    	{\Pro(\HA\fn)\op}
    	\arrow[""{name=0, anchor=center, inner sep=0}, "{\mc U}", curve={height=-10pt}, tail, from=1-1, to=2-2]
    	\arrow["\simeq"{description}, tail reversed, from=1-1, to=3-1]
    	\arrow[""{name=1, anchor=center, inner sep=0}, "\kspec", curve={height=-10pt}, from=2-2, to=1-1]
    	\arrow[""{name=2, anchor=center, inner sep=0}, "{\pf -}", curve={height=-12pt}, from=2-2, to=3-1]
    	\arrow[""{name=3, anchor=center, inner sep=0}, curve={height=-10pt}, tail, from=3-1, to=2-2]
    	\arrow["\dashv"{anchor=center, rotate=-129}, draw=none, from=0, to=1]
    	\arrow["\dashv"{anchor=center, rotate=-45}, draw=none, from=3, to=2]
    \end{tikzcd}\]
\end{theorem}
\begin{proof}
    We first show that the left adjoint of $\kspec : \HA\op \to \IKp$ is $\mc U : \IKp \to \HA\op$, which associates to an image finite poset its Heyting algebra of upward closed subsets. For any Heyting algebra $A\in\HA$ and any $X\in\IKp$, we have
    \begin{align*}
    \IKp(X,\kspec A) 
    &\cong \int_{P\in\Kp}\Set(X(P),\HA(A,\rg P)) \\ 
    &\cong \int_{P\in\Kp}\HA(A,\Set(X(P),\rg P)) \\
    &\cong \HA(A,\int_{P\in\Kp}\Set(X(P),\rg P)) \\
    &\cong \HA(A,\IKp(X,\Omega)) \\
    &\cong \HA(A,\mc UX)
    \end{align*}
    The first isomorphism holds by the end formula of natural transformations; the second holds by currying; the third holds since representable functors preserve ends; the fourth holds by the explicit computation of the subobject classifier in~\cref{subobject}; the final one holds since subobjects of $X$ in $\IKp$ are exactly given by its upward closed subsets. This implies $\mc U$ is indeed the left adjoint of $\kspec$.

    Now under the equivalence $\mc U : \Pro(\HA\fn)\op \simeq \IKp$, this implies the induced left adjoint is simply the inclusion $\Pro(\HA\fn)\op \inj \HA\op$. Notice that the profinite completion functor is left adjoint to the forgetful functor
    \[\begin{tikzcd}
    {\Pro(\HA\fn)} & \HA
    \arrow[""{name=0, anchor=center, inner sep=0}, tail, from=1-1, to=1-2]
    \arrow[""{name=1, anchor=center, inner sep=0}, "{\widehat-}", curve={height=-15pt}, from=1-2, to=1-1]
    \arrow["\dashv"{anchor=center, rotate=90}, draw=none, from=1, to=0]
    \end{tikzcd}\]
    By the uniqueness of adjoints, under the equivalence $\Pro(\HA\fn)\op \simeq \IKp$, the nerve functor $\kspec$ must thus be equivalent to the dual of the profinite completion.
\end{proof}

\begin{remark}[Sheaf representation as profinite completion]
    Now we can fulfil our promise in~\cref{sec:intro}. In~\cite{ghilardi1995sheaf}, one studies a Heyting algebra $A$ by mapping it into the subobject lattice of its sheaf representation,
    \[ A \to \sub_{\mc K}(\kspec A). \]
    By~\cref{sheafrepasprocompletion}, this map is exactly the counit of the adjunction $\mc U : \IKp \leftrightarrows \HA\op : \kspec$. Furthermore, under the equivalence $\IKp \simeq \Pro(\HA\fn)\op$, this counit is equivalently the unit for the adjunction $\pf- : \HA \leftrightarrows \Pro(\HA\fn)$, i.e. the map is given by the profinite completion of the Heyting algebra $A$,
    \[ A \to \sub_{\mc K}(\kspec A) \cong \pf A. \]
    If $A$ is finitely presented, then by the finite model property of IPC, this map will indeed be an embedding. Thus, by the result in~\cite{ghilardi1995sheaf}, uniform interpolation can indeed be obtained via restricting the adjoints between profinite Heyting algebras.
\end{remark}

\begin{remark}[Bellissima's representation of Heyting algebras]
    In fact, the nerve functor $\kspec : \HA\op \to \IKp$ also agrees with Bellissima's representation of finitely generated free Heyting algebras given in~\cite{bellissima1986finitely}. For $H[n]$, \emph{loc.\ cit.}\ constructs a poset, in fact an image finite poset, where $H[n]$ embeds into the upward closed subsets of this poset. The lattice of upward closed subsets of this image finite poset is again the profinite completion of $H[n]$, as explained in~\cite{RN939}. Thus, the nerve functor $\kspec$ is indeed a generalisation of Bellissima's construction.
\end{remark}

\subsection{Nerve for posets}\label{posetnerve}

In fact, another important construction in~\cite{ghilardi1995sheaf} can also be described as a special case of the general nerve construction, which is the nerve for \emph{posets}. Consider the inclusion $\Kp \inj \Pos\fn \hook \Pos$. \cref{creatingfincolimkp} shows this preserves finite colimits. By~\cref{representationinIKp}, we will write the nerve functor as
\[ \uv - : \Pos \to \IKp, \]
where for any poset $L$, the sheaf $\uv L$ is given by 
\[ \uv L(P) \cong \Pos(P,L). \]

\begin{proposition}\label{nerveforposet}
    $\uv - : \Pos \to \IKp$ is the right adjoint of the inclusion functor $\IKp \inj \Pos$.
    \[\begin{tikzcd}
    	\IKp & \Pos
    	\arrow[""{name=0, anchor=center, inner sep=0}, curve={height=-12pt}, tail, from=1-1, to=1-2]
    	\arrow[""{name=1, anchor=center, inner sep=0}, "{\uv-}", from=1-2, to=1-1]
    	\arrow["\dashv"{anchor=center, rotate=-89}, draw=none, from=0, to=1]
    \end{tikzcd}\]
\end{proposition}
\begin{proof}
    By construction, the left adjoint $\IKp \to \Pos$ is computed by the left Kan extension,
    \[\begin{tikzcd}
    	\Kp & \Pos \\
    	\IKp
    	\arrow[tail, from=1-1, to=1-2]
    	\arrow[hook, from=1-1, to=2-1]
    	\arrow["L"', from=2-1, to=1-2]
    \end{tikzcd}\]
    In particular, for any $X\in\IKp$, $LX \cong \ct_{f\in X(P)}P$. But observe from the proof of~\cref{creatingfincolimkp} that the inclusion $\IKp \inj \Pos$ preserves all colimits: it is evident for arbitrary coproducts; for coequalisers, the same proof applies since the proof only depends on maps being open and the fact that for any coequaliser $R \rightrightarrows P$ and for any $p\in P$, the poset $\dv p$ is finite. This implies that the colimit $\ct_{f\in X(P)}P$ computed in $\Pos$ is the same as computed in $\IKp$, thus the left adjoint $L$ is indeed given by the inclusion $\IKp \inj \Pos$.
\end{proof}

\begin{remark}[An explicit construction of the nerve for posets]
    One can mimic the step-by-step construction presented in~\cite{almeida2024colimits} to provide a concrete description of the right adjoint $\uv- : \Pos \to \IKp$ on the level of posets. However, this construction involves a complex transfinite process. The nerve construction presented in this paper is quite straightforward by comparison.
\end{remark}

\begin{remark}[Profinite distributive lattices and profinite Heyting algebras]
    Just as $\IKp$ is the dual category of profinite Heyting algebras, we also have the following equivalences,
    \[ \Pro(\DL\fn)\op \simeq \Ind(\DL\fn\op) \simeq \Ind(\Pos\fn) \simeq \Pos. \]
    In other words, $\Pos$ is the dual category of profinite distributive lattices. As with the case of profinite Heyting algebras, profinite distributive lattices can be concretely described as internal distributive lattices in the category of Stone spaces (cf.~\cite{johnstone1982stone}). In this case, the nerve adjunction in~\cref{nerveforposet} dually corresponds to the following adjunction,
    \[\begin{tikzcd}
    	{\Pro(\HA\fn)} & {\Pro(\DL\fn)}
    	\arrow[""{name=0, anchor=center, inner sep=0}, "U"', tail, from=1-1, to=1-2]
    	\arrow[""{name=1, anchor=center, inner sep=0}, "F"', curve={height=15pt}, from=1-2, to=1-1]
    	\arrow["\dashv"{anchor=center, rotate=-90}, draw=none, from=1, to=0]
    \end{tikzcd}\]
    where $U$ is the forgetful functor and $F$ is the free functor. One can show this adjunction is \emph{monadic}, thus the adjunction between $\IKp$ and $\Pos$ is \emph{comonadic}. 
\end{remark}

\begin{example}[Subobject classifier as nerve of a poset]\label{RN}
  The subobject classifier $\Omega$ can also be identified as $\uv{\Sigma}$ for the Sierpi\'nski poset $\Sigma$. In fact, for any $P\in\Kp$, by~\cref{subobject}
  \[ \Omega(P) \cong \mc UP \cong \Pos(P,\Sigma) \cong \uv{\Sigma}(P). \]
\end{example}

We end this section by summarising the relationship between the nerve for Heyting algebras and nerve for posets:

\begin{proposition}\label{decompspec}
  The nerve adjunction between $\IKp$ and $\DL\op$ can be decomposed as that for $\Pos$ and also that for $\HA\op$ as follows,
  \[\begin{tikzcd}
    & \Pos \\
    \IKp && {\DL\op} \\
    & {\HA\op}
    \arrow[""{name=0, anchor=center, inner sep=0}, "{{\uv-}}"{description}, curve={height=-12pt}, from=1-2, to=2-1]
    \arrow[""{name=1, anchor=center, inner sep=0}, "\rg"{description}, curve={height=-12pt}, from=1-2, to=2-3]
    \arrow[""{name=2, anchor=center, inner sep=0}, curve={height=-12pt}, tail, from=2-1, to=1-2]
    \arrow[""{name=3, anchor=center, inner sep=0}, "\rg"{description}, curve={height=-12pt}, from=2-1, to=3-2]
    \arrow[""{name=4, anchor=center, inner sep=0}, "\spec"{description}, curve={height=-12pt}, from=2-3, to=1-2]
    \arrow[""{name=5, anchor=center, inner sep=0}, "F"{description}, curve={height=-12pt}, from=2-3, to=3-2]
    \arrow[""{name=6, anchor=center, inner sep=0}, "\kspec"{description}, curve={height=-12pt}, from=3-2, to=2-1]
    \arrow[""{name=7, anchor=center, inner sep=0}, "U"{description}, curve={height=-12pt}, from=3-2, to=2-3]
    \arrow["\dashv"{anchor=center, rotate=-46}, draw=none, from=2, to=0]
    \arrow["\dashv"{anchor=center, rotate=-140}, draw=none, from=3, to=6]
    \arrow["\dashv"{anchor=center, rotate=-135}, draw=none, from=4, to=1]
    \arrow["\dashv"{anchor=center, rotate=-38}, draw=none, from=7, to=5]
  \end{tikzcd}\]
  Here $F \dashv U : \HA \to \DL$ is the free-forgetful adjunction between distributive lattices and Heyting algebras; $\spec : \DL\op \to \Pos$ takes a distributive lattice $D$ to the poset $\DL(D,\Sigma)$, and its left adjoint $\rg$ again takes a poset to the lattice of upward closed subsets.
\end{proposition}
\begin{proof}
  We have seen the two composite left adjoints from $\IKp$ to $\DL\op$ both take $X\in\IKp$ to $\rg X$, thus the diagram commutes by uniqueness of adjoints.
\end{proof}

\begin{remark}[Nerve for free Heyting algebras as nerve for posets]
    Let $D$ be a distributive lattice and let $FD$ be the free Heyting algebra generated by $D$. For any $P\in\Kp$, we can compute that
    \[ \kspec FD (P) \cong \HA(FD,\mc UP) \cong \DL(D,\mc UP) \cong \Pos(P,\spec D) \cong \uv{\spec D}(P). \]
    Thus, the nerve $\kspec FD$ for a \emph{free} Heyting algebra is isomorphic to the nerve of the \emph{poset} $\spec D$. In particular, for a finite poset $L$, $\uv L$ can be used to present the finitely generated free Heyting algebra over $\mc UL$. This is the basis for analysing general finitely presented Heyting algebras in~\cite{ghilardi1995sheaf}.
\end{remark}

\section{Internal logic of the K-topos and properties of uniform interpolation}\label{sec:internaluniform}

To apply the nerve construction given in~\cref{sec:nerveonK} to the study of uniform interpolation, or more generally adjoints of maps between profinite Heyting algebras, we first explain how these adjoints can be described in the K-topos $\mc K$. Let $A$ be a profinite Heyting algebra and let $A[\pf{\ms x}]$ be the free profinite Heyting algebra over $A$ generated by one object. Categorically, $A[\pf{\ms x}]$ can be realised as the following coproduct in $\Pro(\HA\fn)$
\[ A[\pf{\ms x}] \cong H[\ms x] \pf{\oplus} A, \]
where $\pf{\oplus}$ denote the coproduct in $\Pro(\HA\fn)$. Dually, $A[\pf{\ms x}]$ is computed by a product in $\mc K$,
\[ \mc J(A[\pf{\ms x}]) \cong \mc J(H[\ms x] \pf{\oplus} A) \cong \Omega \times \mc JA. \]
This holds since $\Pro(\HA\fn)\op \hook\mc K$ is a right adjoint, and thus it takes colimits in $\Pro(\HA\fn)$ to limits in $\mc K$, together with~\cref{subobjectasanerve} that $\mc JH[\ms x] \cong \Omega$. Furthermore, the canonical inclusion $A \inj A[\pf{\ms x}]$ now corresponds to the projection
\[ \pi : \Omega \times \mc J A \to \mc J A. \]
The adjoints of this inclusion can then be described as external existential and universal quantifiers in the K-topos $\mc K$:

\begin{lemma}\label{subobjadjoint}
    For any profinite Heyting algebra $A$, the canonical inclusion $A \inj A[\pf{\ms x}]$ corresponds to the pullback of subobjects along the projection $\pi : \Omega \times \mc JA \to \mc JA$, thus its left and right adjoints correspond to the existential and universal quantifiers on subobjects,
    \[\begin{tikzcd}
    	{A} & {A[\pf{\ms x}]} \\
    	{\sub_{\mc K}(\mc JA)} & {\sub_{\mc K}(\Omega \times \mc JA)}
    	\arrow[""{name=0, anchor=center, inner sep=0}, tail, from=1-1, to=1-2]
    	\arrow["\cong"', from=1-1, to=2-1]
    	\arrow[""{name=1, anchor=center, inner sep=0}, "\Forall"{description}, curve={height=-12pt}, from=1-2, to=1-1]
    	\arrow[""{name=2, anchor=center, inner sep=0}, "\Exists"{description}, curve={height=12pt}, from=1-2, to=1-1]
    	\arrow["\cong", from=1-2, to=2-2]
    	\arrow[""{name=3, anchor=center, inner sep=0}, "{\pi\inv}"{description}, tail, from=2-1, to=2-2]
    	\arrow[""{name=4, anchor=center, inner sep=0}, "{\forall_{\pi}}"{description}, curve={height=-12pt}, from=2-2, to=2-1]
    	\arrow[""{name=5, anchor=center, inner sep=0}, "{\exists_{\pi}}"{description}, curve={height=12pt}, from=2-2, to=2-1]
    	\arrow["\dashv"{anchor=center, rotate=-90}, draw=none, from=0, to=1]
    	\arrow["\dashv"{anchor=center, rotate=-90}, draw=none, from=2, to=0]
    	\arrow["\dashv"{anchor=center, rotate=-90}, draw=none, from=3, to=4]
    	\arrow["\dashv"{anchor=center, rotate=-90}, draw=none, from=5, to=3]
    \end{tikzcd}\]
\end{lemma}
\begin{proof}
    This is a direct consequence of~\cref{leftadjointofnerveIKp}.
\end{proof}

The adjoints in~\cref{subobjadjoint} between subobjects of powers of $\Omega$ can be \emph{internalised} in $\mc K$. For any $X\in\mc K$, let $\mc P(X) \coloneq \Omega^X$ be the power object in $\mc K$. Then we have the following \emph{internal adjunctions} in $\mc K$:
\[\begin{tikzcd}
	{\mc P(\mc J A)} & {\mc P(\Omega \times \mc JA)}
	\arrow[""{name=0, anchor=center, inner sep=0}, "{{\pi\inv}}"{description}, tail, from=1-1, to=1-2]
	\arrow[""{name=1, anchor=center, inner sep=0}, "{{\forall_{\pi}}}", curve={height=-12pt}, from=1-2, to=1-1]
	\arrow[""{name=2, anchor=center, inner sep=0}, "{{\exists_{\pi}}}"', curve={height=12pt}, from=1-2, to=1-1]
	\arrow["\dashv"{anchor=center, rotate=-90}, draw=none, from=0, to=1]
	\arrow["\dashv"{anchor=center, rotate=-90}, draw=none, from=2, to=0]
\end{tikzcd}\]
We may thus ask whether an external property for $\Exists,\Forall$ between profinite Heyting algebras holds internally for the maps $\exists_\pi,\forall_\pi : \mc P(\Omega \times \mc JA) \to \mc P(\mc JA)$ in $\mc K$. As mentioned in~\cref{sec:intro}, we will mainly look at two properties, which will be discussed in~\cref{internallyinjective} and~\cref{irreduciblesubobj}.

However, before that let us comment on our general strategy. It is clear from above that these internal adjoints are quantifying over $\Omega$. We will thus first establish various properties of $\Omega$ in the internal language of $\mc K$. This relies on the internal logic of the topos, including Kripke-Joyal semantics (see e.g.~\cite[Sec. VI.6]{maclane2012sheaves}) and also the stack semantics~\cite{shulman2010stack}, which translates internal properties in a topos to externally verifiable statements about the site generating the topos. 

After establishing the relevant internal properties of $\Omega$, we will then proceed with an argument directly in the \emph{internal language} of $\mc K$, to establish the corresponding properties of these adjoints quantifying over $\Omega$. It is well-known that constructive reasoning is valid in the internal language of any topos, thus any constructive consequences of the established properties of $\Omega$ will also be valid in the internal logic of $\mc K$.

From an external perspective, a proof in the internal language amounts to a \emph{parametrised proof} in all the slices. Thus, for instance, declaring a proposition $\varphi$ in the internal language amounts externally to declare an arbitrary \emph{generalised element} $\varphi : X \to \Omega$. Hence in principle, an internal argument can be faithfully translated into the external language. However, an internal language argument could sometimes be more streamlined since one can pretend to work with ordinary sets and functions, as long as one remembers to use only constructive principles.

The reason we choose to prove the properties of these adjoints in two steps, viz.\ first by proving some general properties of $\Omega$, and then apply it in an internal argument to derive properties about the adjoints, rather than directly, is simply that we think the internal properties of $\Omega$ are of independent interests themselves, and could potentially find other applications outside the scope of uniform interpolation.

\subsection{Propositions are internally projective and the subobject classifier is internally injective}\label{internallyinjective}

In this subsection we establish the dual Frobenius property mentioned in~\cref{sec:intro}. From a topos-theoretic perspective, this is connected to \emph{projectivity of propositions} and \emph{injectivity of $\Omega$} in $\mc K$.

\begin{definition}[Internal projectivity]\label{def:internalproj}
    In a topos $\mc E$, an object $X$ is \emph{internally projective} if the right adjoint $\prod_X : \mc E/X \to \mc E$ of the pullback preserves epis, or equivalently the exponential $(-)^X : \mc E \to \mc E$ preserves epis. 
\end{definition}

In fact, internal projectivity can also be stated in an extension of the Kripke-Joyal interpretation of the internal language of $\mc K$, called the \emph{stack semantics}, where now one can quantify over \emph{all} objects in a topos; see~\cite{shulman2010stack}. The statement that ``every epimorphism over $X$ splits'', interpreted via the stack semantics, is \emph{equivalent} to the formulation of $X$ being internally projective in~\cref{def:internalproj} (cf.~\cite[D4.5.3]{johnstone2002sketches}).

Hence, we can state projectivity in the internal language of $\mc K$, and the following shows that internally in $\mc K$, all propositions are projective:

\begin{proposition}\label{propositionsareprojective}
    The following holds in the internal logic of $\mc K$,
    \[ \mc K \models \fa p\Omega p \text{ is projective.} \]
    Explicitly, this means for any $P\in\Kp$ and subobject $U\in\sub_{\mc K}(P) \cong \rg P$, the pushforward $\prod_\iota : \mc K/U \to \mc K/P$ preserves epis for the inclusion $\iota : U \hook P$. 
\end{proposition}
\begin{proof}
    Let $f : X \surj Y$ be an epi over $U$. To show $\prod_\iota f : \prod_\iota X \to \prod_\iota Y$ is an epi, we need to show that for any $Q \to P$ in $\Kp$ and any $Q \to \prod_\iota Y$, there exists some local section over $P$ as shown below,
    \[
    \begin{tikzcd}
        R \ar[d, two heads, dashed] \ar[r, dashed] & \prod_\iota X \ar[d, "\prod_\iota f"] \\
        Q \ar[r] & \prod_\iota Y
    \end{tikzcd}
    \]
    By the universal property, a map $Q \to \prod_\iota Y$ over $P$ is equivalently a map $Q|_U \to Y$ over $U$, where $Q|_U \cong Q \times_P U$ is the pullback, which is simply the preimage of $U$ in $Q$. By the fact that $X \surj Y$ is an epi, we do get a local section below,
    \[
    \begin{tikzcd}
        R' \ar[d, "q'"', two heads] \ar[r] & X \ar[d, two heads, "f"] \\
        Q|_U \ar[r] & Y
    \end{tikzcd}
    \]
    where $R'\in\Kp$ is some finite poset with an open surjection onto $Q|_U$. This way, we can construct $R$ as follows: as a set let $R \coloneq R' \coprod Q\backslash Q|_U$, where the orders within $R'$ and $Q\backslash Q|_U$ are inherited from the respective posets. Between them, we specify that for any $r\in R'$ and $x\in Q\backslash Q|_U$, 
    \[ x \le r \eff x \le q'r, \]
    and no other relations hold. It is easy to verify that $R$ is a well-defined finite poset, with a canonical surjection $q : R \surj Q$, where $qr = q'r$ for $r\in R'$ and $qx = x$ for $x\in Q\backslash Q|_U$. Furthermore, $q$ by construction is open because $q'$ is open. The final result of this construction is that we now have the following pullback squares
    \[
    \begin{tikzcd}
        R' \ar[dr, pullback] \ar[d, "q'"', two heads] \ar[r, hook] & R \ar[d, "q", two heads] \\
        Q|_U \ar[dr, pullback] \ar[r, hook] \ar[d] & Q \ar[d] \\
        U \ar[r, hook] & P
    \end{tikzcd}
    \]
    Thus by the universal property again, a section $R' \to X$ over $U$ now extends to a section $R \to \prod_\iota X$ over $P$. This implies that $\prod_\iota X \to \prod_\iota Y$ indeed admits local sections, thus is an epi. Hence, $U$ as a subterminal object is internally projective in $\mc K/P$.
\end{proof}

On the other hand, there is also a dual notion of \emph{internal injectivity}; similarly, the formulation in~\cref{def:internalinj} coincides with stating injectivity in the stack semantics (cf.~\cite{blechschmidt2018flabby}):

\begin{definition}[Internal injectivity]\label{def:internalinj}
    In a topos $\mc E$, an object $X$ is internally injective if $X^- : \mc E\op \to \mc E$ takes monos to epis.
\end{definition}

But in the case of injectivity, we can pass freely from external injectivity to internal injectivity (also cf.~\cite[Thm. 3.8]{blechschmidt2018flabby}). In particular, it is well-known that $\Omega$ in any topos is externally injective (cf.~\cite[A2.2.6]{johnstone2002sketches}), thus it is also internally injective in $\mc K$.

\begin{proposition}
    Any externally injective object $X$ in a topos is also \emph{internally injective}.
\end{proposition}
\begin{proof}
    Let $A \hook B$ be a mono. We show that for any object $C$ and map $C \to X^A$, there exists a lift as indicated below on the left,
    \[\begin{tikzcd}
    	& {X^B} & {A \times C} & X \\
    	C & {X^A} & {B \times C}
    	\arrow[from=1-2, to=2-2]
    	\arrow[from=1-3, to=1-4]
    	\arrow[hook, from=1-3, to=2-3]
    	\arrow[dashed, from=2-1, to=1-2]
    	\arrow[from=2-1, to=2-2]
    	\arrow[dashed, from=2-3, to=1-4]
    \end{tikzcd}\]
    Given such a solid diagram, by transposing along the exponentials, we obtain the solid diagram on the right above. By injectivity of $X$ there exists an extension as the dashed arrow on the right above, thus it transposes to a lift indicated by the dashed arrow on the left above. This implies that $X^B \to X^A$ is in fact a \emph{split epi}, by considering the lift of the identity map on $X^A$.
\end{proof}

\begin{theorem}\label{internalinjFrob}
    For any object $X$ in $\mc K$, the following holds in the internal logic of $\mc K$,
    \begin{align*}
        \mc K \models &\,\fa{\varphi}{\mc P(X)} \fa{\psi}{\mc P(X \times \Omega)} \fa x{X} \\
        &\,(\varphi(x) \to \ex p{\Omega} \psi(x,p)) \to \ex p{\Omega}(\varphi(x) \to \psi(x,p))
    \end{align*}
\end{theorem}
\begin{proof}
    We reason constructively in the internal logic of $\mc K$. Suppose we have $\varphi\in\mc P(X)$ and $\psi \in \mc P(X\times\Omega)$, and assume for some $x\in X$, $\varphi(x) \to \ex p{\Omega}\psi(x,p)$ holds. By~\cref{propositionsareprojective} $\varphi(x)$ is projective, thus this lifts to a map $p : \varphi(x) \to \scomp{p\in\Omega}{\psi(x,p)}$, i.e.\ we have the solid part of the following diagram,
    \[
    \begin{tikzcd}
        \varphi(x) \ar[d, hook] \ar[r, "p"] & \scomp{p\in\Omega}{\psi(x,p)} \ar[d, hook] \\
        1 \ar[r, dashed, "q"'] & \Omega
    \end{tikzcd}
    \]
    By injectivity of $\Omega$, we can find an extension $q : 1 \to \Omega$ of $p$. In this case, $\varphi(x) \to \psi(x,q)$ because by construction $\varphi(x) \to p = q$. Thus, $\ex p{\Omega}(\varphi(x) \to \psi(x,p))$ holds.
\end{proof}

\begin{corollary}\label{dualfrobenius}
    For any profinite Heyting algebra $A$, the left adjoint $\Exists : A[\pf{\ms x}] \to A$ of the canonical inclusion $\iota : A \inj A[\pf{\ms x}]$ satisfies the dual Frobenius property: for any $\varphi\in A$ and any $\psi \in A[\pf{\ms x}]$,
    \[ \Exists(\iota\varphi \to \psi) = \varphi \to \Exists\psi. \]
    In particular, the dual Frobenius property holds for the left uniform interpolation quantifier of finitely presented Heyting algebras.
\end{corollary}
\begin{proof}
    The claim on profinite Heyting algebras follows straightforwardly from the observation in~\cref{subobjadjoint}, and the internal property in~\cref{internalinjFrob}: the dual Frobenius property is exactly the externalisation of the internal statement in~\cref{internalinjFrob} when $X$ is $\mc JA$. For the claim on finitely presented Heyting algebras, recall that the uniform interpolation is the restriction of the corresponding adjoints between their profinite completion. In particular, for any finitely presented Heyting algebra $A$, we have
    \[\begin{tikzcd}
    	A & {A[\ms x]} \\
    	{\pf A} & {\pf{A[\ms x]} \cong \pf{A}[\pf{\ms x}]}
    	\arrow[""{name=0, anchor=center, inner sep=0}, "\iota"{description}, tail, from=1-1, to=1-2]
    	\arrow[tail, from=1-1, to=2-1]
    	\arrow[""{name=1, anchor=center, inner sep=0}, "\Exists"{description}, curve={height=15pt}, from=1-2, to=1-1]
    	\arrow[tail, from=1-2, to=2-2]
    	\arrow[""{name=2, anchor=center, inner sep=0}, "{\pf{\iota}}"', tail, from=2-1, to=2-2]
    	\arrow[""{name=3, anchor=center, inner sep=0}, "\Exists"{description}, curve={height=15pt}, from=2-2, to=2-1]
    	\arrow["\dashv"{anchor=center, rotate=-90}, draw=none, from=1, to=0]
    	\arrow["\dashv"{anchor=center, rotate=-90}, draw=none, from=3, to=2]
    \end{tikzcd}\]
    The isomorphism at the bottom right holds since $\pf-$ as a left adjoint preserves coproducts,
    \[ \pf{A[\ms x]} \cong \pf{H[\ms x] \oplus A} \cong H[\ms x] \pf{\oplus} \pf A \cong \pf A[\pf{\ms x}]. \]
    Hence the dual Frobenius property also holds for the uniform interpolation quantifiers.
\end{proof}

\subsection{The subobject classifier is internally irreducible}\label{irreduciblesubobj}

Next, we will establish that the relevant right adjoints of maps between profinite Heyting algebras in fact preserve finite joins. As we will see, this is a consequence of $\Omega$ being \emph{internally irreducible} in $\mc K$. However, unlike the case of injectivity which holds for $\Omega$ in a general topos, this is a property special to $\mc K$. To show this we first need some preliminary results on Heyting algebras.

\begin{definition}[Local Heyting algebra]
    A Heyting algebra $A$ is \emph{local} if it is non-trivial and for any $a,b\in A$, $a \vee b = 1$ iff $a = 1$ or $b = 1$.
\end{definition}

\begin{proposition}\label{finitelocalpolylocal}
    Let $A$ be a Heyting algebra. If $A$ is local, then so is $A[\ms x]$.
\end{proposition}
\begin{proof}
    By uniform interpolation for IPC, the inclusion $A \inj A[\ms x]$ has a \emph{right adjoint} $\ms A : A[\ms x] \to A$ (cf.~\cite[Cor. 15]{pitts1992uniform}). Since $A \inj A[\ms x]$ is an inclusion, $\Forall$ will be a retract, i.e.\ $\Forall a = a$ for $a\in A$. We then identify $A$ as a subalgebra of $A[\ms x]$. We apply a glueing argument along this right adjoint, which is the following comma square,
    \[\begin{tikzcd}
    	{\qsi{A[\ms x]}} & A \\
    	{A[\ms x]} & A
    	\arrow[""{name=0, anchor=center, inner sep=0}, "t", from=1-1, to=1-2]
    	\arrow[""{name=1, anchor=center, inner sep=0}, "r", from=1-1, to=2-1]
    	\arrow[equals, from=1-2, to=2-2]
    	\arrow["s", curve={height=-12pt}, dashed, from=2-1, to=1-1]
    	\arrow["\Forall"', from=2-1, to=2-2]
    	\arrow[curve={height=-6pt}, between={0.2}{0.8}, Rightarrow, from=0, to=1]
    \end{tikzcd}\]
    Since $\Forall$ is a right adjoint, it preserves all finite meets. Thus the comma object is the lattice-theoretic construction of Artin glueing (cf.~\cite{WRAITH1974345}). This way, $\qsi{A[\ms x]}$ is again a Heyting algebra and $r$ is a Heyting algebra morphism. Elements in $\qsi{A[\ms x]}$ are simply pairs $(\varphi,a)$ such that $a \le \Forall\varphi$. The maps $r,t$ are simply projections. The meets and joins in $\qsi{A[\ms x]}$ are pointwise, while the implication is given by
    \[ (\varphi,a) \to (\psi,b) = (\varphi \to \psi,(a\to b)\wedge\Forall(\varphi \to \psi)). \]
    This way, the map $a \mapsto (a,a)$ is a Heyting algebra morphism $A \to \qsi{A[\ms x]}$, thus by the universal property of $A[\ms x]$, the element $(\ms x,0)\in\qsi{A[\ms x]}$ induces a section $s$ of $r$. Since this is a section, for any $\varphi\in A[\ms x]$ we also write $s\varphi = (\varphi,s_0\varphi)$.

    Now we proceed to show $A[\ms x]$ is local. Evidently $0 \neq 1$ in $A[\ms x]$ if $A$ is non-trivial. Suppose $\varphi \vee \psi = 1$. This way,
    \[ s_0(\varphi \vee \psi) = s_0\varphi \vee s_0\psi = 1. \]
    But since $A$ is local, it follows that $s_0\varphi = 1$ or $s_0\psi = 1$. Without loss of generality say $s_0\varphi = 1$, then by construction we have 
    \[ s_0\varphi = 1 \le \Forall\varphi \nt \Forall\varphi = 1. \]
    In particular, $1 \le \varphi$ in $A[\ms x]$, thus $\varphi = 1$. 
\end{proof}

\begin{remark}
    The proof of~\cref{finitelocalpolylocal} clearly relies on the existence of the right adjoint $\Forall : A[\ms x] \to A$, which as we have mentioned is a consequence of uniform interpolation for IPC. However, if $A$ is finite, then the existence of this right adjoint is automatic, since it can be defined by a finite join. In particular, in~\cref{internalirredOmega} below we will use~\cref{finitelocalpolylocal} but only for finite Heyting algebras, and hence there is no actual dependency on the proof of uniform interpolation for IPC.
\end{remark}

\begin{theorem}\label{internalirredOmega}
    $\Omega$ in $\mc K$ is internally irreducible: $\Omega$ is inhabited, and
    \[ \mc K \models \fa{\varphi,\psi}{\mc P(\Omega)} \varphi \cup \psi = \Omega \to \varphi = \Omega \vee \psi = \Omega. \]
\end{theorem}
\begin{proof}
    $\Omega$ being inhabited is evident, since it has a global point. For irreducibility under binary union, by the Kripke-Joyal semantics (cf.~\cite{maclane2012sheaves}), it suffices to show for any $P\in\Kpp$ and any $\varphi,\psi\in\mc P(\Omega)(P)$, we have
    \[ \varphi \vee \psi = 1 \nt \varphi = 1 \vee \psi = 1. \]
    In other words, we need to show $\mc P(\Omega)(P)$ is a local Heyting algebra for any $P\in\Kpp$. Notice that we have the following isomorphisms,
    \[ \mc P(\Omega)(P) \cong \mc K(P \times \Omega,\Omega) \cong \mc U(P \times \Omega). \]
    By the duality $\mc U : \IKp \simeq \Pro(\HA\fn)\op$, we have
    \[ \mc U(P \times \Omega) \cong \mc UP \pf{\oplus} \mc U\Omega \cong \mc UP \pf{\oplus} H[\ms x] \cong \mc UP[\pf{\ms x}] \cong \pf{\mc UP[\ms x]}. \]
    Now notice that $\mc UP$ is local since $P\in\Kpp$ is rooted. $\mc UP$ is a finite Heyting algebra thus the right adjoint $\Forall : \mc UP[\ms x] \to \mc UP$ exists. Thus by~\cref{finitelocalpolylocal}, $\mc UP[\ms x]$ is also local. Since $\mc UP[\ms x]$ is finitely presented, there exists some $n$ and $\varphi\in H[n]$ that 
    \[ H[n] \surj H[n]/\varphi \cong \mc UP[\ms x]. \]
    Since $\mc UP[\ms x]$ is local, $\varphi$ must be join-irreducible in $H[n]$. By~\cite[Cor. 7.2]{RN939}, this implies $\varphi\in\pf{H[n]}$ is also join-irreducible. Observe that $\pf{\mc UP[\ms x]}$ is also the following quotient,
    \[ \pf{\mc UP[\ms x]} \cong \pf{H[n]/\varphi} \cong \pf{H[n]}/\varphi. \]
    This way, it follows that $\pf{\mc UP[\ms x]} \cong \mc P(\Omega)(P)$ is indeed local.
\end{proof}

\begin{remark}[An alternative way to show irreducibility]
    In the proof of~\cref{internalirredOmega} we cite a result in~\cite{RN939}, that essentially allows us to show for a finitely presented Heyting algebra $A$, if $A$ is local then so is its profinite completion $\pf A$. However, we can also show the result in~\cref{internalirredOmega} via a direct translation using Kripke-Joyal semantics to a statement about finite posets.
\end{remark}

\begin{corollary}\label{rightadjlocal}
    For any profinite Heyting algebra $A$, the right adjoint $\Forall : A[\pf{\ms x}] \to A$ to the canonical inclusion $A \inj A[\pf{\ms x}]$ preserves finite joins. In particular, the right uniform interpolant for finitely presented Heyting algebras commutes with finite joins.
\end{corollary}
\begin{proof}
    We first show for any object $X\in\mc K$, it holds internally in $\mc K$ that
    \begin{align*}
        \mc K \models &\,\fa{\varphi,\psi}{\mc P(X \times \Omega)}\fa xX \\
        &\,\fa p{\Omega}(\varphi(x,p)\vee\psi(x,p)) \to (\fa p\Omega \varphi(x,p) \vee \fa p\Omega \psi(x,p)).
    \end{align*}
    We reason constructively in the internal logic of $\mc K$. Suppose we have $\varphi,\psi \subseteq X \times \Omega$ and $x\in X$. Consider the following subsets,
    \[ U = \scomp{p:\Omega}{\varphi(x,p)}, \quad V = \scomp{p:\Omega}{\psi(x,p)}. \]
    Now by assumption $U \cup V = \Omega$. By~\cref{internalirredOmega}, $U = \Omega$ or $V = \Omega$, thus the desired claim holds. Again by~\cref{subobjadjoint}, the externalisation of this internal property for $X = \mc JA$ exactly shows $\Forall : A[\pf{\ms x}] \to A$ preserves finite joins. The claim on uniform interpolation follows from an argument similar to that in~\cref{dualfrobenius}.
\end{proof}

\section{Subtopoi of the K-topos}\label{sec:subtopoi}

In this section we digress a bit from our main investigation and classify all subtopoi of $\mc K$. Our strategy is to calculate explicitly those maps $j : \Omega \to \Omega$ in $\mc K$ that constitute Lawvere-Tierney topologies. As a start, we can easily classify all endomorphisms of $\Omega$:

\begin{lemma}\label{subisfree}
  $\mc K(\Omega,\Omega) \cong H[\ms x]$. 
\end{lemma}
\begin{proof}
  This follows from the fact that $\sub_{\mc K} \cong \mc K(-,\Omega) : \mc K \to \Pro(\HA\fn)\op$ is the left adjoint to the inclusion $\mc J : \Pro(\HA\fn)\op \hook\mc K$ by~\cref{leftadjointofnerveIKp}, and the result in~\cref{subobjectasanerve} that $\Omega \cong \mc JH[\ms x]$.
\end{proof}

Thus, our goal is to characterise those elements in $H[\ms x]$ that induce Lawvere-Tierney topologies via the isomorphism $\mc K(\Omega,\Omega) \cong H[\ms x]$. 

\begin{lemma}\label{lttopology}
  There is a bijective correspondence between subtopoi of $\mc K$ and polynomials $\varphi(\ms x) \in H[\ms x]$ satisfying the following axioms:
  \begin{enumerate}[(i)]
    \item $\varphi$ \emph{preserves top}: $\varphi(1) = 1$ in $H[\ms x]$.
    \item $\varphi$ \emph{preserves conjunction}: $\varphi(\ms x\wedge\ms y) = \varphi(\ms x) \wedge \varphi(\ms y)$ in $H[\ms x,\ms y]$.
    \item $\varphi$ is \emph{idempotent}: $\varphi(\varphi(\ms x)) = \varphi(\ms x)$ in $H[\ms x]$.
  \end{enumerate}
\end{lemma}
\begin{proof}
  Recall a subtopos of $\mc K$ can be identified with a Lawvere-Tierney topology on $\mc K$, i.e. a map $\varphi : \um \to \um$ such that
  \[\begin{tikzcd}
    1 & \um & \um & \um & {\um \times \um} & {\um \times \um} \\
    & \um && \um & \um & \um
    \arrow["\top", from=1-1, to=1-2]
    \arrow["\top"', from=1-1, to=2-2]
    \arrow["\varphi", from=1-2, to=2-2]
    \arrow["\varphi", from=1-3, to=1-4]
    \arrow["\varphi"', from=1-3, to=2-4]
    \arrow["\varphi", from=1-4, to=2-4]
    \arrow["{\varphi \times \varphi}", from=1-5, to=1-6]
    \arrow["\wedge"', from=1-5, to=2-5]
    \arrow["\wedge", from=1-6, to=2-6]
    \arrow["\varphi"', from=2-5, to=2-6]
  \end{tikzcd}\]
  By~\cref{subisfree}, a map $\varphi : \um \to \um$ is equivalently a polynomial $\varphi\in H[\ms x]$. Note we know that $\um \cong \mc JH[\ms x]$ and $\um \times \um \cong \mc J(H[\ms x] \pf{\oplus} H[\ms y]) \cong \mc J(\pf{H[\ms x,\ms y]})$. The above diagrams commute iff the corresponding diagrams commute below,
  \[\begin{tikzcd}[column sep = 4ex]
    2 & {H[\ms x]} & {H[\ms x]} & {H[\ms x]} & {\widehat{H[\ms x,\ms y]}} && {\widehat{H[\ms x,\ms y]}} \\
    & {H[\ms x]} && {H[\ms x]} & {H[\ms x]} && {H[\ms x]}
    \arrow["\ms x \mapsto 1"', from=1-2, to=1-1]
    \arrow["{\ms x \mapsto \varphi}"', from=1-4, to=1-3]
    \arrow["{\ms x,\ms y \mapsto \varphi(\ms x),\varphi(\ms y)}"', from=1-7, to=1-5]
    \arrow["\ms x \mapsto 1", from=2-2, to=1-1]
    \arrow["{\ms x \mapsto \varphi}"{description}, from=2-2, to=1-2]
    \arrow["{\ms x \mapsto \varphi}", from=2-4, to=1-3]
    \arrow["{\ms x \mapsto \varphi}"{description}, from=2-4, to=1-4]
    \arrow["{\ms x \mapsto \ms x \wedge \ms y}"{description}, from=2-5, to=1-5]
    \arrow["{\ms x \mapsto \ms x \wedge \ms y}"{description}, from=2-7, to=1-7]
    \arrow["{\ms x \mapsto \varphi}", from=2-7, to=2-5]
  \end{tikzcd}\]
  The first two diagrams directly translate to (i) and (ii). For the correspondence between (iii) and the diagram to the right, it suffices to note that these maps going into the profinite completion $\widehat{H[\ms x,\ms y]}$ factor through $H[\ms x,\ms y]$. 
\end{proof}

We are then able to classify the subtopoi of $\mc K$ through such polynomials:

\begin{theorem}\label{thm:subtoposofK}
    The only non-trivial subtopos of $\mc K$ is the subtopos of $\neg\neg$-sheaves $\mc K_{\neg\neg}$, whose direct image by~\cref{nabladoublenegation} corresponds to $\nabla : \Set \to \mc K$ mentioned in~\cref{Kcohesive}.
\end{theorem}
\begin{proof}
  We first show the only polynomials in $H[\ms x]$ satisfying (i--iii) in~\cref{lttopology} are $\ms x,\neg\neg \ms x$ and the constant polynomial $1$ in $H[\ms x]$. Thus, besides $\mc K$ itself and the trivial topos, $\mc K$ only has one non-trivial subtopos, which is the subtopos of $\neg\neg$-sheaves $\mc K_{\neg\neg} \hook \mc K$.
  
  To classify polynomials in $H[\ms x]$ satisfying (i--iii), we recall from~\cite{nishimura1960formulas} that $H[\ms x]$ can be explicitly described by the following Hasse diagram,
    \[\begin{tikzcd}[column sep = 0ex, row sep = small]
    	&& 1 & \\
    	& \vdots && \vdots \\
    	{\neg\neg\ms x \to \ms x} && {\neg\ms x \vee \neg\neg\ms x} \\
    	& {\ms x \vee \neg\ms x} && {\neg\neg\ms x} \\
    	{\neg\ms x} && {\ms x} \\
    	& 0
    	\arrow[from=2-2, to=1-3]
    	\arrow[from=2-4, to=1-3]
    	\arrow[from=3-1, to=2-2]
    	\arrow[from=3-3, to=2-2]
    	\arrow[from=3-3, to=2-4]
    	\arrow[from=4-2, to=3-1]
    	\arrow[from=4-2, to=3-3]
    	\arrow[from=4-4, to=3-3]
    	\arrow[from=5-1, to=4-2]
    	\arrow[from=5-3, to=4-2]
    	\arrow[from=5-3, to=4-4]
    	\arrow[from=6-2, to=5-1]
    	\arrow[from=6-2, to=5-3]
    \end{tikzcd}\]
  In particular, the order on the polynomials reflects their action on Heyting algebras: if $\varphi \le \psi$ in $H[\ms x]$, then for any Heyting algebra $A$ and $a\in A$, $\varphi(a) \le \psi(a)$. This way, first note that for any $\varphi\in H[\ms x]$, $\varphi(1) = 1$ iff $\ms x \le \varphi$ in $H[\ms x]$. Now if $\ms x \vee \neg\ms x \le \varphi$, we show $\varphi$ does not preserve conjunction, unless $\varphi$ is the constant $1$. In fact, we show it cannot be \emph{monotone}. Firstly, we observe that
  \[ (\ms x \vee \neg\ms x)(0) = 0 \vee \neg 0 = 1, \]
  which implies that $\varphi(0) = 1$ for all $\ms x \vee \neg\ms x \le \varphi$. However, since $0 \le \ms x$, $\varphi(\ms x) = \varphi$, which implies $\varphi$ is \emph{not} monotone if $\varphi \neq 1$. This way, besides $1$, the only possible choices of $\varphi$ that $\ms x \le \varphi$ and $\ms x \vee \neg\ms x \not\le \varphi$ are $\ms x$ and $\neg\neg \ms x$. 
\end{proof}

\begin{remark}
    It is also possible to provide a direct proof of~\cref{thm:subtoposofK} via considering the Grothendieck topologies finer than the canonical topology $J$ introduced in~\cref{topologyonKp}. For this one may first show that any topology $K$ finer than $J$ is completely characterised by which open inclusions it contains: for any sieve $S$ on $P$, let $Q\hook P$ be the joint image of $S$ in $P$, and $S$ is a $K$-covering iff $Q \hook P$ is a $K$-covering. Now if the only open inclusions that are $K$-coverings are the identities, $K = J$. If $K$ contains an open inclusion $Q \hook P$ such that $Q$ does not contain all maximal points of $P$, then by looking at the pullback of $1 \to P$ outside the image of $Q$, we conclude that $K$ contains the empty covering on $1$; this corresponds to the degenerate subtopos. The only remaining case is when $K$ contains some, equivalently all, non-trivial open inclusions that contain all the maximal points; this corresponds to the $\neg\neg$-topology. However, writing down the details of the above proof sketch requires a lot more combinatorial arguments on finite posets, and we find our algebraic proof in~\cref{thm:subtoposofK} much cleaner.
\end{remark}

\section{The K-topos as an exact completion}\label{sec:exactcomp}

In this last section we will prove one of the main results in this paper, which is that $\mc K$ is the ex/reg-completion of the dual category of profinite Heyting algebras. Here the ex/reg-completion is the universal embedding of a regular category into an exact category (cf.~\cite{RN769}). For convenience, in this section we will mainly work with the dual category $\IKp$. We will first show that the inclusion $\IKp \hook \mc K$ identifies it as a reflective subcategory closed under small coproducts and image factorisation. In particular, $\IKp$ will be an infinitary extensive regular category, and these structures are created by the inclusion $\IKp \hook \mc K$. We will then show that the additional objects in $\mc K$ can all be written as effective quotients in $\IKp$, thus delivering the main result. 

\begin{remark}
    Before we proceed to the proof, it is instructive to mention that the result in this section is somewhat surprising from a categorical perspective. The phenomenon that the exact completion of a category forms a topos is not new to category theory, and in fact lies at the heart of the construction of realisability topoi. These constructions follow a general pattern, e.g. as shown in~\cite{menni2000exact} that the ex/reg-completion of any locally cartesian closed regular category with a generic mono forms a topos. However, $\IKp$ is \emph{not} locally cartesian closed. In fact it cannot be cartesian closed. For this recall from~\cref{subobjectasanerve} that $\IKp$ has a subobject classifier. It is also complete and cocomplete as it is a reflective subcategory of $\mc K$ (cf.~\cref{IKpreflective}). Hence, if $\IKp$ were cartesian closed, then it will in particular be an \emph{elementary topos}, which implies it will be exact. As we will see in~\cref{notexact}, this is not the case, and we do need to formally add all the missing quotients to $\IKp$ to obtain a topos, in this case the K-topos $\mc K$.
\end{remark}

We first show $\IKp$ is closed under coproducts:

\begin{lemma}\label{IKppreservescoproducts}
  $\IKp$ is closed under small coproducts in $\mc K$, hence it is infinitary extensive.
\end{lemma}
\begin{proof}
  Here we use the site presentation $(\Kpp,J)$ of $\mc K$. Suppose we have a family $\set{X_i}_{i\in I}$ in $\IKp$. Let $\bigsqcup_{i\in I}X_i$ be their disjoint union. It is easy to see that this is their coproduct in $\IKp$. Now for any $P\in\Kpp$,
  \[ \IKp(P,\bigsqcup_{i\in I}X_i) \cong \coprod_{i\in I}\IKp(P,X_i), \]
  since $P$ by definition is rooted, thus any map $P \to \bigsqcup_{i\in I}X_i$ factors through a unique $X_i$. By~\cref{unionofsub}, coproducts in $\mc K$ are computed pointwise on $\Kpp$, thus this shows the disjoint union $\bigsqcup_{i\in I}X_i$ is exactly the coproduct $\coprod_{i\in I}X_i$ in $\mc K$, hence $\IKp$ is closed under coproducts in $\mc K$.
\end{proof}

For image factorisation, we first show that epis in $\IKp$ and $\mc K$ agree:

\begin{lemma}\label{IKppreservesregularepi}
  The inclusion $\IKp \hook \mc K$ preserves and reflects epis. In particular, $\IKp$ is closed under image factorisation in $\mc K$, hence is a regular category. 
\end{lemma}
\begin{proof}
  Since $\IKp \hook \mc K$ is fully faithful, it reflects epis. An epi in $\IKp$ is simply a surjective map $f : Y \surj X$. In this case we show it is also an epi in $\mc K$, which means locally we can find sections of $f$. Let $P\in\Kpp$ and let $x : P \to X$ be a map. To find a local section of $f$ at $P$, we may factor $x$ into $P \surj \dv x(*) \hook X$ as below. Since $f$ is surjective, we can find $y\in Y$ that $fy = x(*)$. This way we obtain the following diagram,
  \[\begin{tikzcd}
    {P \otimes_{\dv x(*)}\dv y} & {\dv y} & Y \\
    P & {\dv x(*)} & X
    \arrow[two heads, from=1-1, to=1-2]
    \arrow[two heads, from=1-1, to=2-1]
    \arrow[hook, from=1-2, to=1-3]
    \arrow["f", two heads, from=1-2, to=2-2]
    \arrow["f", two heads, from=1-3, to=2-3]
    \arrow[two heads, from=2-1, to=2-2]
    \arrow[hook, from=2-2, to=2-3]
  \end{tikzcd}\]
  which shows $f$ admits local sections, thus is an epimorphism in $\mc K$. Now given $g : Y \to X$ in $\IKp$, let $Z$ be the set-theoretic image of $g$. Since $g$ is open, it is factored into a surjection followed by an open embedding in $\IKp$,
  \[
  \begin{tikzcd}
    Y \ar[r, two heads, "h"] & Z \ar[r, hook] & X
  \end{tikzcd}
  \]
  Since $h$ is also an epi in $\mc K$, this will also be the image factorisation in $\mc K$.
\end{proof}

Furthermore, we show filtered colimits in $\IKp$ also agree with those in $\mc K$.

\begin{lemma}\label{IKppreservefiltercolim}
    $\IKp \hook \mc K$ preserves filtered colimits, i.e. $\IKp$ is an accessible reflective subcategory of $\mc K$.
\end{lemma}
\begin{proof}
    By~\cref{filteredIKp}, as an ind-completion $\IKp \simeq \Lex(\Kp\op,\Set)$, filtered colimits in $\IKp$ are computed pointwise as functors. Thus, this agrees with the filtered colimits in the presheaf topos $[\Kp\op,\Set]$, hence also in $\mc K$.
\end{proof}

Thus, $\IKp$ is an infinitary extensive regular category. If $\IKp$ were also \emph{exact}, then by Giraud's theorem $\IKp$ would be a topos since it has a small generating family, say $\Kp$. However, $\IKp$ is \emph{not} exact. The crucial point here is that given an equivalence relation $R \rightrightarrows X$ in $\IKp$, its coequaliser in $\IKp$ is \emph{not} computed as that in $\mc K$. This means $\IKp$ is missing some quotients in $\mc K$. This can already be seen in the fact that the Yoneda embedding $\Kp \hook \mc K$ does not preserve coequalisers:

\begin{remark}[$\IKp$ is not exact]\label{notexact}
  We first show that the Yoneda embedding $\Kp \hook \mc K$ does not preserve coequalisers, via the criterion given in~\cite[Prop. 2.4.8]{RN900}. Consider two maps $f,g : \Sigma \to \Sigma^2$ where $f(0) = (0,1)$ and $g(0) = (1,0)$, with $f(1) = g(1) = (1,1)$. Now the two maps are open, and their coequaliser in $\Kp$ is $P = \set{0 < x < 1}$ with the evident projection $h : \Sigma^2 \surj P$. $\Sigma \rightrightarrows \Sigma^2 \surj P$ is a coequaliser in $\mc K$ iff for any pair of maps $u,v : Q \rightrightarrows \Sigma^2$ that $hu = hv$, there is a surjection $Q' \surj Q$ that factors $u,v$ as below,
  \[
  \begin{tikzcd}
    Q' \ar[d, dashed] \ar[r, two heads, dashed] & Q \ar[d, shift right, "u"'] \ar[d, shift left, "v"] \\
    \Sigma \ar[r, shift left, "f"] \ar[r, shift right, "g"'] & \Sigma^2 \ar[r, two heads] & P
  \end{tikzcd}
  \]
  However, note that the element $(0,0)\in\Sigma^2$ is not in the joint image of $f,g$. As long as $u,v$ map some element in $Q$ to $(0,0)$, such factorisation cannot exist. For instance, $Q$ can be taken to be $\Sigma^2$ itself, with $u$ the identity and $v$ the map interchanging $(0,1)$ and $(1,0)$. 

  Now let $R$ be the smallest equivalence relation on $\Sigma^2$ in $\IKp$ that contains the pair $\pair{f,g}$. The computation of $R$ in $\IKp$ agrees with that in $\mc K$, since $\IKp$ is closed under limits and also subobjects by~\cref{closedundersubobjects}. Furthermore, the generation of the smallest equivalence relation does not change the computation of coequaliser in any category. This way, $P$ will remain as the coequaliser of $R$ in $\IKp$, and we show that $R$ is not effective, i.e.\ the kernel pair of $P$ cannot be $R$. For this let $P'$ be the coequaliser of $R$ in $\mc K$. By the argument above, we have a non-trivial quotient $P' \surj P$ in $\mc K$. Since $\mc K$ as a topos is exact, $R$ is the kernel pair of $P'$, which shows that $R$ is strictly smaller than the kernel pair of $P$.
  
  We also mention that the argument showing $\IKp$ is not exact can also be done directly, without referring to the topos $\mc K$, following a similar argument as in~\cite{de2025finite}. However, finite product of image finite posets in $\IKp$ is hard to describe explicitly, as it involves a further transfinite process after taking the cartesian product of two posets as shown in~\cite{almeida2024colimits}. Thus, explicitly writing down the generated equivalence relation $R$ requires some non-trivial computation, and we find our categorical argument minimal for the task.
\end{remark}

To help the reader build intuition on ex/reg-completion, we here mention a more well-known instance.

\begin{remark}[Intuition of ex/reg-completion]
    Let $\Stone$ and $\KHaus$ be the category of Stone spaces and compact Hausdorff spaces, respectively. The canonical inclusion $\Stone \hook \KHaus$ is again an ex/reg-completion (cf.~\cite{marra2020characterisation}). Though $\Stone$ is cocomplete, not all equivalence relations have effective quotients. For instance, the closed interval $[0,1]$ admits a close surjection from the Cantor space $2^{\mbb N} \surj [0,1]$, whose kernel pair $R \rightrightarrows 2^{\mbb N}$ will be an equivalence relation in $\Stone$. However, the coequaliser of $R \rightrightarrows 2^{\mbb N}$ in $\Stone$ is the singleton space $1$, since $[0,1]$ is connected. This implies that this equivalence relation in $\Stone$ is not effective. By adding effective quotients to all equivalence relations in $\Stone$, one obtains all compact Hausdorff spaces. 
\end{remark}

Similarly, though $\IKp$ is cocomplete, not all equivalence relations in $\IKp$ have effective quotients. By adding these missing quotients to $\IKp$, one obtains the K-topos $\mc K$. To show this, we first show the following lemma, which is a general observation about ex/reg-completion that appears to be new to us:

\begin{lemma}\label{lem:subobjectexreg}
    Let $\mc C$ be a regular category, and let $\iota : \mc C \hook \mc D$ be a fully faithful regular embedding into an exact category, such that $\mc C$ is closed under subobjects in $\mc D$. If every object in $\mc D$ admits an epimorphism from $\iota X$ for some $X\in \mc C$, $\mc D$ is the ex/reg-completion of $\mc C$.
\end{lemma}
\begin{proof}
    According to~\cite[Def. 11]{RN769}, $\mc C\exreg$ can be concretely described as follows:
    \begin{itemize}
        \item objects are equivalence relations $R \hook X \times X$ in $\mc C$;
        \item a morphism $f : (X,R) \to (Y,T)$ is a relation $f \inj X \times Y$ that is functional for the equivalence classes on $R$ and $T$.
    \end{itemize}
    By the universal property of ex/reg-completion, there is a uniquely induced exact functor $\mc C\exreg \to \mc D$, which takes a pair $(X,R)$ to the coequaliser of $\iota R \rightrightarrows \iota X$ in $\mc D$. By assumption, for any $A\in\mc D$ there exists an epi $\iota X \surj A$ in $\mc D$. Its kernel pair $B \hook \iota X \times \iota X \cong \iota(X \times X)$ by assumption comes from an equivalence relation $R \hook X \times X$, thus $\mc C\exreg \to \mc D$ must be essentially surjective. 
    
    To show fully faithfulness, let $A,B$ be the images of $(X,R)$ and $(Y,T)$, respectively. We claim that maps between $A,B$ in $\mc D$ exactly correspond to relations from $\iota X$ to $\iota Y$ that are functional w.r.t.\ the equivalence relations $\iota R$ and $\iota T$. Note that fully faithfulness follows from this claim and the concrete description of morphisms in $\mc C\exreg$ above. This is because by assumption, subobjects of $\iota X \times \iota Y \cong \iota(X \times Y)$ in $\mc D$ are in bijection with subobjects of $X \times Y$ in $\mc C$, and that being functional w.r.t.\ two equivalence relations is preserved and reflected by a fully faithful regular functor. But the claim on maps between $A,B$ follows easily from the fact that the ex/reg-completion is \emph{idempotent} (cf.~\cite{RN769}). In particular, $\mc D \simeq \mc D\exreg$ since $\mc D$ is exact, and the inverse $\mc D\exreg \to \mc D$, which is induced by universal property from the fact that $\mc D$ is exact, again sends the pair $(\iota X,\iota R)$ and $(\iota Y,\iota T)$ to $A,B$, respectively. Hence, the desired claim again holds by the description of morphisms in the ex/reg-completion above.
\end{proof}

\begin{theorem}\label{maintheorem}
    The inclusion $\IKp \hook \mc K$ identifies $\mc K$ as the ex/reg-completion of $\IKp$.
\end{theorem}
\begin{proof}
    By~\cref{closedundersubobjects}, $\IKp$ is closed under subobjects in $\mc K$, thus by~\cref{lem:subobjectexreg} it suffices to show for any $F\in\mc K$, there exists $X\in\IKp$ that $X \surj F$. Given an arbitrary sheaf $F\in\mc K$, we can easily construct a quotient as below,
    \[ \coprod_{x\in F(P),P\in\Kpp}P \surj F. \]
    This map is an epimorphism since for any local section $x : P \to F$ for $P\in\Kpp$, it lifts to the domain by construction,
    \[
    \begin{tikzcd}
        & \coprod_{x\in F(P),P\in\Kpp} P \ar[d] \\ 
        P \ar[r, "x"'] \ar[ur, dashed] & F
    \end{tikzcd}
    \]
    By~\cref{IKppreservescoproducts}, the coproduct indeed belongs to $\IKp$, which completes the proof.
\end{proof}

\vspace{.3pt}
\subsection*{Acknowledgements}

We are thankful to \emph{Rodrigo Nicolau Almeida} for discussing with us an early draft of this work. We are thankful to \emph{Matteo De Berardinis} for letting us know an error in an early draft of this paper. We are thankful to \emph{Matías Menni} for some useful comments and also for letting us know an error we made in a first version of this paper. We are also thankful to two anonymous referees, whose comments improved the readability and the accuracy of this paper.

\subsection*{Funding}

The author has no relevant funding information to disclose.

\bibliography{mybib}
\bibliographystyle{alpha}

\end{document}